\documentclass[11pt]{amsart}
\usepackage[utf8]{inputenc}
\usepackage{amsmath}
\usepackage{amsfonts}
\usepackage{amssymb}
\usepackage{amsthm}
\usepackage{csquotes}
\usepackage[english]{babel}
\usepackage[shortlabels]{enumitem}
\usepackage{lmodern}
\usepackage{cite}

\usepackage{color}

\usepackage{tikz}
\usetikzlibrary{decorations.markings}
\usetikzlibrary{patterns}

\newtheorem{theorem}{Theorem}[section]

\newtheorem{corollary}[theorem]{Corollary}
\newtheorem{lemma}[theorem]{Lemma}
\newtheorem{proposition}[theorem]{Proposition}

\theoremstyle{definition}
\newtheorem{definition}[theorem]{Definition}

\theoremstyle{remark} \theoremstyle{remark}
\newtheorem{remark}[theorem]{Remark}
\newtheorem{example}[theorem]{Example}

\newtheorem*{theorem*}{Theorem}

\DeclareMathOperator{\ad}{ad}

\DeclareMathOperator{\id}{id}
\DeclareMathOperator{\End}{End}
\DeclareMathOperator{\diag}{diag}

\newcommand{\mft}{{\mathfrak t}}

\newcommand{\CC}{{\mathbb C}}
\newcommand{\SO}{{\mathrm{SO}}}
\newcommand{\Sp}{{\mathrm{Sp}}}
\newcommand{\Gr}{{\mathrm{Gr}}}
\newcommand{\Un}{{\mathrm{U}}}
\newcommand{\SU}{{\mathrm{SU}}}

\newcommand{\HH}{{\mathbb{H}}}
\newcommand{\RR}{{\mathbb{R}}}
\newcommand{\ZZ}{{\mathbb{Z}}}

\numberwithin{equation}{section}

\begin{document}

\title{GKM actions on almost quaternionic manifolds}
 
 \author{Oliver Goertsches}
  \address{Philipps-Universit\"at Marburg, Fachbereich Mathematik und Informatik, Hans-Meerwein-Straße 6, 35043 Marburg, Germany}
  \email{goertsch@mathematik.uni-marburg.de}
  
 \author{Eugenia Loiudice}
 \email{loiudice@mathematik.uni-marburg.de}

\begin{abstract}
We introduce quaternionic structures on abstract GKM graphs, as the combinatorial counterpart of almost quaternionic structures left invariant by a torus action of GKM type. In the GKM$_3$ setting the $2$-faces of the GKM graph can naturally be divided into quaternionic and complex $2$-faces; it turns out that for GKM$_3$ actions on positive quaternion-K\"ahler manifolds the quaternionic $2$-faces are biangles or triangles, and the complex $2$-faces triangles or quadrangles. We show purely combinatorially that any abstract GKM$_3$ graph with quaternionic structure satisfying this restriction on the $2$-faces of the GKM graph is that of a torus action on  quaternionic projective space $\HH P^n$ or the Grassmannian $\Gr_2(\CC^n)$ of complex $2$-planes in $\CC^n$. 
\end{abstract}


\maketitle

\section{Introduction} 

In GKM theory, named after Goresky--Kottwitz--MacPherson \cite{GKM}, one associates to certain torus actions on smooth manifolds a labelled graph that encodes a variety of topological information, both on the action as well as on the space that is acted on. More precisely, one considers orientable manifolds whose cohomology (in this paper with rational coefficients) vanishes in odd degrees, with actions that have finite fixed point set, such that the orbit space of the one-skeleton of the action is homeomorphic to a graph. The theory can be considered as a generalization of the (quasi)toric setting, and applies to many geometrically relevant situations, like equal rank homogeneous spaces \cite{GHZ} or certain equal rank biquotients \cite{GKZbiquot, GKZHamdim6}.

The GKM graph of the action is the orbit space of the one-skeleton mentioned above, equipped with a labeling of the edges given by weights of the isotropy representation. In the general setting these weights are elements in the integer lattice in ${\mathfrak{t}}^*$ modulo $\pm 1$, where ${\mathfrak{t}}$ is the Lie algebra of the acting torus $T$. Additional geometric structures, may leave their mark on the GKM graph, as summarized in \cite[Section 2]{GKMbundles}. Most prominently, given an invariant almost complex structure, the weights may be regarded as elements in ${\mathfrak{t}}^*$; sometimes one speaks about a signed structure on the GKM graph. The existence of invariant symplectic or even K\"ahler structures poses additional combinatorial restrictions on the graph.

In this paper we consider GKM actions that leave invariant an almost quaternionic structure, i.e., a rank $3$ subbundle $Q\subset \End(TM)$ locally spanned by three almost complex structures that satisfy the quaternion relations. The special structure of the isotropy representations in this setting, see Proposition \ref{prop:quatrep}, results in what we define to be a \emph{quaternionic structure} on a GKM graph. Roughly speaking, it consists of a so-called \emph{quaternionic weight} at each vertex, which is the weight of the induced representation on the corresponding fiber of $Q$, a grouping of the edges at each vertex in \emph{quaternionic pairs}, as well as certain compatibility conditions for adjacent vertices, expressed in terms of a compatible connection. For the precise conditions, see Definition \ref{defn:quatstrgraph}. 

The most prominent examples of almost quaternionic manifolds are qua\-ternion-K\"ahler manifolds. Conjecturally, any positive quaternion-Kähler manifold is a Wolf space, and thus in particular an equal rank homogeneous space. Hence, all known examples of positive quaternion-Kähler manifolds admit actions of GKM type. In Section \ref{sec:thewolfspaces} we describe explicitly the GKM graphs, including their quaternionic structure, of hyperbolic projective space $\HH P^n$ and the Grassmannian $\Gr_2(\CC^n)$ of $2$-planes in $\CC^n$.

In most of this article, we will restrict to GKM$_3$ actions and graphs, so that we may speak about $2$-faces of the GKM graph. 
Our approach is comparable to that of \cite{GW}, where a rather different geometric structures was considered: there, it was understood that the $2$-faces of isometric GKM$_3$ actions on orientable positively curved Riemannian manifolds are necessarily biangles and triangles; using purely combinatorial arguments, it was shown that an abstract GKM$_3$ graph with this restriction on the $2$-faces is already that of an action on a compact rank one symmetric space. 

In the almost quaternionic GKM$_3$ setting, we call those $2$-faces  \emph{quaternionic} that correspond to quaternionic pairs of edges; the other ones we call \emph{complex} because we show in Proposition \ref{prop:2facescomplex} that they admit a signed structure compatible with the quaternionic structure. In the quaternion-Kähler setting, we observe that the quaternionic $2$-faces are biangles or triangles, while the complex $2$-faces are either triangles or quadrangles. This motivates the set of assumptions of our entirely combinatorial main result, which completely describes the quaternionic GKM$_3$ graphs with this restriction on the $2$-faces:

\begin{theorem}
Let $\Gamma$ be an abstract GKM$_3$ graph with a quaternionic struture, such that all of its quaternionic $2$-faces are biangles or noncomplex triangles, and all complex $2$-faces are triangles or quadrangles. Then $\Gamma$ is the GKM graph of a torus action on $\HH P^n$ or $\Gr_2(\CC^n)$ by quaternion-K\"ahler automorphisms.
\end{theorem}

As a geometric corollary, we obtain the statement (see Corollary \ref{cor:main}) that a positive quaternion-K\"ahler manifold that admits a torus action of type GKM$_3$ has the rational cohomology ring of either $\HH P^n$ or $\Gr_2(\CC^{n})$. In case of $\Gr_2(\CC^{n})$, this forces the quaternion-K\"ahler manifold to be the Grassmannian by \cite[Theorem 0.2]{LeBrunSalamon}. These geometric consequences, however, were already known: using methods from complex algebraic geometry, \cite[Theorem 5.1]{Occhetta} shows that positive quaternion-K\"ahler manifolds that admit a torus action such that certain extremal fixed point components are isolated are automatically Wolf spaces. In this sense, our result does not give new evidence for the LeBrun--Salamon conjecture, but should be regarded as a combinatorial version of known rigidity statements for quaternionic structures.\\

%{\color{blue} cf proposition 2.3, 2.4 of ON POSITIVE QUATERNIONIC KAEHLER MANIFOLDS WITH B4=1, KIM - LEE}

\noindent {\bf Acknowledgements.} We wish to thank Nicolina Istrati, S\"onke Rollenske, and Uwe Semmelmann for helpful discussions. We are grateful to Michael Wiemeler for pointing out to us the results of reference \cite{Occhetta}, and for fixing the gap in \cite{GW}, see Footnote \ref{footnote}.

\section{Quaternionic geometries}\label{sec:quaternionicgeometries}

Let $M$ be a $4n$-dimensional smooth manifold. 
An \emph{almost quaternionic} structure on $M$ is a rank $3$ subbundle $Q\subset \End (TM)$ locally spanned by almost hypercomplex structures, namely by a triple of (locally defined) almost complex structures $(J_1,J_2,J_3)$ satisfying the quaternionic identities
\[
 J_1J_2=-J_2J_1=J_3.
\]

Although almost quaternionic manifolds in general do not admit a globally defined almost complex structure, their twistor space is naturally an almost complex manifold. We recall that the twistor space of an almost quaternionic manifold is the $2$-sphere bundle $Z\rightarrow M$ of $(Q, \langle\,,\,\rangle)$ where $\langle\,,\,\rangle$ is the inner product on $Q$ defined considering any local quaternionic base $(J_1,J_2,J_3)$ of $Q$ as an orthonormal base.
Observe that every almost quaternionic connection $\nabla$, which is a connection on $M$ leaving $Q$ invariant (i.e., for every vector field $X$ on $M$ and for every section $\xi$ of $Q$ we have that $\nabla_X\xi$ is again a section of $Q$) induces a decomposition of $TZ$ into horizontal and vertical subbundles. Since any element $y$ in $Z$ is a complex structure on $T_xM$, where $x=\pi(y)$, we have a complex structure $J^H_y$ on the horizontal subspace $H_y$ of $T_yZ$ and thus the horizontal subbundle $H$ of $TZ$ is a complex vector bundle. Moreover since the fibers of $\pi$ are unit spheres, the vertical subspace of $T_yZ$ is
\[
 V_y=\{J\in Q_x\,|\, \langle J,y\rangle=0\}
\]
and we can consider on it the complex structure $J_y^V$ defined by
\[
 J_y^V(J)=y\circ J.
\]
Then the vertical subbundle is also a complex bundle over $M$ and 
$$J^{\nabla}=J^H\oplus J^V$$
defines an almost complex structure on $Z$ (depending on the given connection $\nabla$). In \cite[Theorem 3.1]{AMP} the authors proved that this almost complex structure on $Z$ depends only on the torsion of the almost quaternionic connection. This construction generalizes the $4$-dimensional and the quaternionic case considered respectively in \cite{AHS}, \cite{S}.

\vspace{0.2 cm} 
We recall that an Oproiu connection on $M$ is an almost quaternionic connection whose torsion equals the structure tensor $T^Q$ of $(M,Q)$; locally, in the domain of any almost hypercomplex structure $B=(J_1,J_2,J_3)$, this tensor is defined by
\begin{equation}\label{oproiutorsion}
 T^Q=T^B+\sum_{\alpha=1}^3\partial(\tau_\alpha^B\otimes J_\alpha),
\end{equation}
where $\partial$ is the alternation operator, and the $\tau^B_\alpha$'s and $T^B$ denote the structure one-forms and the structure tensor of $B$:
\[
 \tau_\alpha^B(X)=\frac{1}{4n-2}{\mathrm{tr}} (J_\alpha T^B(X,\cdot))
\]
\[T^B=\frac{1}{12}\sum_{\alpha=1}^3[[J_\alpha,J_\alpha]]\] 
where the Nijenhuis tensor of $J_\alpha$ is given by
\[
 [[J_\alpha,J_\alpha]](X,Y)=2 ([J_\alpha X,J_\alpha Y]+J_\alpha^2[X,Y]-J_\alpha[J_\alpha X,Y]-J_\alpha[X,J_\alpha Y]),
\]
for all vector fields $X,Y$ on $M$. The structure tensor $T^Q$ is independent of the choice of local base $B$ of $Q$ (see \cite{oproiu}). Considering an Oproiu connection on $M$ 
then gives, via the construction above, a natural almost complex structure $J^Q$ on $Z$ depending only on $Q$.

\vspace{0.3 cm}
In the Riemannian setting, if an almost quaternionic structure $Q$ on a $4n$-dimensional Riemannian manifold $M$, $n\geq 2$, is preserved by the Levi-Civita connection, then $M$ is called a \emph{quaternion-K\"ahler} manifold. This condition is equivalent to the holonomy group of $(M^{4n},g)$ being a subgroup of $\Sp(n)\Sp(1):=\Sp(n)\times \Sp(1)/\mathbb{Z}_2$ (see for example \cite[Proposition 14.36]{Besse}). A quaternion-Kähler manifold is called \emph{positive} if its Riemannian metric is complete and its scalar curvature is positive.

\vspace{0.3 cm}
Notice that $4$-dimensional Riemannian manifold are excluded from this definition. Indeed, the holonomy group of any oriented $4$-dimensional Riemannian manifold is contained in $\SO(4)=\Sp(1)\Sp(1)$, and hence we have no condition in this case. We recall in the next example that every oriented Riemannian manifold of dimension $4$ (acted on by a compact Lie group $G$) is a ($G$-invariant) almost quaternionic manifold. 

\vspace{0.3 cm} 
We say that an oriented $4$-dimensional Riemanian manifold is \emph{quaternion-K\"ahler} if it is Einstein with nonzero scalar curvature and the Weyl tensor $W$ is self-dual. The definition of $4$-dimensional quaternion-K\"ahler manifold can be understood looking at its associated twistor space. Indeed, the self-duality condition of $W$ is equivalent to the integrability of the almost complex structure $J^Q$ on the twistor space of $M$, see Theorem~13.46 of \cite{Besse} or also \cite{Penrose} and \cite{AHS}. Moreover in dimension $>4$ it well known that any quaternion-K\"ahler manifold is Einstein (see \cite{Berger}) and that its twistor space is a complex manifold, see \cite{Salamon}.

% 
% On the other hand, for the twistor space of a Riemannian $4$-dimensional orientable manifold to be complex contact it is necessary that $M$ is Einstein with nonzero scalar curvature and admits an orientation with respect to which the Weyl tensor $W$ of $g$ is self-dual, namely  $\star W=W$, where $\star$ is the Hodge star operator \cite{}[Ward], [Atiyah, Hitchin,Singer].
% 
% We say that an oriented Riemaniann $4$-dimensional manifold is \emph{quaternion-K\"ahler} if it is Einstein with nonzero scalar curvature and the Weyl tensor $W$ of $g$ is self-dual.

\begin{example}\label{ex:4diminvalmostquat}
Consider any isometric action of a compact connected Lie group $G$ on a $4$-dimensional oriented Riemannian manifold $M$. We claim that $M$ admits a $G$-invariant almost quaternionic structure.

Consider the $G$-equivariant rank $6$ vector bundle ${\mathfrak{so}}(TM)\to M$ of skew-symmetric endomorphisms of $TM$.  The Riemannian metric induces a $G$-equivariant isomorphism ${\mathfrak{so}}(TM)\cong \Lambda^2(T^*M)$ to the bundle of $2$-forms on $M$. Let $Q_\pm\subset {\mathfrak{so}}(TM)$ denote the subbundles that correspond, under this isomorphism, to the bundles of self-dual and anti-self-dual $2$-forms (i.e., the $\pm 1$-eigenspaces of the Hodge $\ast$-operator). As the Hodge $\ast$-operator commutes with orientation-preserving isometries, these eigenspaces, and hence $Q_\pm$, are $G$-invariant.

Now, both bundles $Q_\pm$ define $G$-invariant almost quaternionic structures on $M$. The bundle $Q_+$ is compatible with the given orientation in the sense that this orientation is the same as the one that is induced by the almost complex structures in $Q_+$.  %Indeed, this can be seen directly by observing that the above decomposition corresponds to the unique decomposition of a skew-symmetric $4\times 4$-matrix as a sum of the form
%\[
%\left(\begin{matrix} 0 & -a & -b & -c \\
%a & 0 & -c & b \\
%b & c &  0 & -a \\
%c & -b & a & 0 \end{matrix} \right) + \left(\begin{matrix} 0 & -x & -y & -z \\
%x & 0 & z & -y \\
%y & -z & 0 & x \\
%z & y & -x & 0
%\end{matrix} \right)
%\]
%\textcolor{red}{coordinate-free argument?!? decomposition taken from MathStackexchange. maybe cite?!?}
% https://math.stackexchange.com/questions/1650187/so4-is-isomorphic-to-so3so3
%and computing that the square of a summand equals $-(a^2+b^2+c^2)$ respectively $-(x^2+y^2+z^2)$ times the identity.

Fiberwise, the set of orthogonal almost complex structures is isomorphic to ${\mathrm{O}}(4)/{\mathrm{U}}(2)$, which is a union of two $2$-spheres, which are therefore necessarily the two $2$-spheres of almost complex structures in $Q_\pm$.  Hence, any orthogonal almost complex structure on $M$ is automatically given by a section of either $Q_+$ or $Q_-$, depending on whether its induced orientation is the one of $M$ or not. (Those almost complex structures which are sections of $Q_+$ are called \emph{compatible}.) 
\end{example}

In general, a (locally defined) almost complex structure on a almost quaternionic manifold $M$ is said to be \emph{compatible} if it is a section of $Z$. Note that if $M$ is positive quaternion-Kähler, then it does not admit any global compatible almost complex structure by \cite[Theorem 3.8]{AMP1}\footnote{\label{footnote1} Note that the argument in \cite[Proposition 3.9]{AMP1}, which is used for the proof of \cite[Theorem 3.8]{AMP1}, does not apply in dimension $4$. The nonexistence of a compatible almost complex structure on $\CC P^2$ can be seen for instance by showing via the long exact homotopy sequence of the twistor fibration $\SU(3)/T^2\to \CC P^2$ that it does not admit a section: the map $\pi_5(\SU(3)/T^2)\to \pi_5(\CC P^2)$ is the homomorphism $\ZZ \overset{\cdot 2}\to \ZZ$ which does not have a section.}.

% \vspace{0.3 cm}
% A Riemannian manifold $(M^{4n},g)$, $n\geq 2$, endowed with an almost quaternionic structure $Q$ which is preserved by the Levi-Civita connection $\nabla$ (i.e., for every $\xi$ smooth section of $Q$ and for every vector field $X$ of $M$, $\nabla_X \xi$ is again a section of $Q$) is called a \emph{quaternion-K\"ahler} manifold.
% 
% This condition is equivalent to the fact that the holonomy group of $(M^{4n},g)$ is a subgroup of $\Sp(n)\Sp(1):=\Sp(n)\times \Sp(1)/\mathbb{Z}_2$ (see \cite{}).
% 
% Notice that $4$-dimensional Riemannian manifold $(M,g)$ are excluded from this definition. Indeed, the holonomy group of any oriented $4-$dimensional Riemannian manifold is $\SO(4)=\Sp(1)\Sp(1)$, and hence we have no-condition when $n=1$.

\section{GKM actions and GKM graphs}\label{sec:gkmactions}

In this paper, we will consider actions of compact tori $T=S^1\times \cdots \times S^1$ on compact connected manifolds $M$.  The \emph{(Borel model of) equivariant cohomology} of such an action is defined as 
\[
H^*_T(M):=H^*(M\times_T ET),
\]
where $ET\to BT$ is the classifying bundle of $T$, and $M\times_T BT$, the so--called \emph{Borel construction}, is defined as the quotient of $M\times ET$ by the (free) diagonal $T$-action. We will only consider rational coefficients, and suppress them in the notation. Via the canonical projection $M\times_T ET \to BT$, the equivariant cohomology $H^*_T(M)$ becomes an algebra over the polynomial ring $R:=H^*(BT)$. The action is called \emph{equivariantly formal} if $H^*_T(M)$ is free as an $R$-module.
\begin{example}
If all odd Betti numbers of $M$ vanish, then any $T$-action on $M$ is equivariantly formal. By \cite[Theorem 6.6]{Salamon} this holds true for positive quaternion-K\"ahler manifolds.
\end{example} 
For equivariantly formal actions, the ordinary cohomology ring is encoded in the equivariant cohomology algebra, see e.g.\ \cite[Theorem 3.10.4 and Corollary 4.2.3]{AlldayPuppe}: 
\begin{proposition}\label{prop:eqformalcomputable}
If the $T$-action on $M$ is equivariantly formal, then the inclusion $M\to M\times_T ET$ of any fiber of $M\times_T ET\to BT$ induces a ring isomorphism
\[
H^*(M)\cong H^*_T(M)/R^+\cdot H^*_T(M)
\]
where $R^+\subset R$ denotes the subspace of polynomials with vanishing constant term.
\end{proposition}

GKM theory, named after Goresky--Kottwitz--MacPherson \cite{GKM}, now poses additional assumptions on the action that make the equivariant cohomology, and hence by Proposition \ref{prop:eqformalcomputable} also the ordinary cohomology, explicitly computable. We will give a brief introduction to the topic; for further details we refer to \cite{GKMsurvey}.

\begin{definition}\label{defn:GKMkaction} The $T$-action on $M$ is called GKM$_k$, for some integer $k\geq 2$, if $M$ is orientable, the odd Betti numbers of $M$ vanish, the fixed point set is finite, and at each fixed point any $k$ weights of the isotropy representation are linearly independent. 

For $k=2$, we simply speak about a GKM action.
\end{definition}
We stress that in our situation we do not have any almost complex or symplectic structure on $M$ that is invariant under the action. Hence, the weights are elements in $\mft^*$ that are only well-defined up to sign. Still, it is reasonable to speak about their linear independence.

Given a GKM $T$-action on $M$, the orbit space $M_1/T$ of the \emph{one-skeleton}
\[
M_1 = \{p\in M\mid \dim T\cdot p \leq 1\}
\]
of the action is homeomorphic to a graph, with vertices corresponding to the fixed points, and edges corresponding to $T$-invariant $2$-spheres -- with the two fixed points in the $2$-sphere being the end points of the edge. Any edge $e$ may be labeled with the weight $\alpha(e)\in \mft^*/\pm 1$ of the isotropy representation of any of the two end points whose weight space is tangent to the sphere.
\begin{definition}
The labeled graph thus obtained is called the \emph{GKM graph} of the GKM action.
\end{definition}

From a more combinatorial point of view, the notion of abstract GKM graph was introduced (for signed GKM graphs, see Definition \ref{defn:signedstructure} below) in \cite{GZ1, GZ2}. We consider an $n$-valent finite graph $\Gamma$. We allow multiple edges between vertices, but no loops. We denote the vertex set of $\Gamma$ by $V(\Gamma)$ and its edge set by $E(\Gamma)$. The edges of $\Gamma$ do not come with a fixed orientation, but we denote by $\tilde{E}(\Gamma)$ the set of edges of $\Gamma$ with orientation. For an oriented edge $e\in \tilde{E}(\Gamma)$ we denote the same edge with opposite orientation by $\bar{e}\in \tilde{E}(\Gamma)$. If we consider an oriented edge $e\in \tilde{E}(\Gamma)$, then we may speak about its initial and terminal vertex $i(e)$ and $t(e)$. The set of oriented edges starting at a vertex $v$ will be denoted $E_v$. 
\begin{definition}
A \emph{connection} on $\Gamma$ is a collection of bijections $\nabla_e:E_{i(e)}\to E_{t(e)}$, for each edge $e\in \tilde{E}(\Gamma)$, satisfying
\begin{enumerate}
\item $\nabla_e e = \bar{e}$
\item $(\nabla_e)^{-1} = \nabla_{\bar{e}}$.
\end{enumerate}
\end{definition}
\begin{definition}
An \emph{abstract GKM graph} $(\Gamma,\alpha)$ is the data of an $n$-valent finite graph $\Gamma$, together with an \emph{axial function} $\alpha:E(\Gamma)\to {\mathbb{Z}}^m/\pm 1$ on the edges of $\Gamma$, such that
\begin{enumerate}
\item For any vertex $v\in V(\Gamma)$ and any  $e,f\in E_v$, $e\neq f$, we have that $\alpha(e)$ and $\alpha(f)$ are linearly independent.
\item There is a connection $\nabla$ on $\Gamma$ \emph{compatible with $\alpha$}, i.e., for any $v\in V(\Gamma)$ and $e,f\in E_v$ there exist lifts $\tilde{\alpha}(f)$ of $\alpha(f)$ and $\tilde{\alpha}(\nabla_ef)$ of $\alpha(\nabla_e f)$ to ${\mathbb{Z}}^m$ such that
\[
\tilde{\alpha}(\nabla_e f) \in \tilde{\alpha}(f) + {\mathbb{Z}}\, \alpha(e).
\]
\end{enumerate}
\end{definition}
It was shown in \cite[p.\ 6]{GZ2} and \cite[Proposition 2.3]{GWnonneg} that the GKM graph of a GKM action always admits a compatible connection, and hence is an abstract GKM graph. Here we identify the integral lattice in ${\mathfrak{t}}^*$ with ${\mathbb{Z}}^m$.

Analogously to Definition \ref{defn:GKMkaction}, an abstract GKM graph is said to be GKM$_k$ if for every vertex and every $k$ edges $e_1,\ldots,e_k\in E_v$ the elements $\alpha(e_1),\ldots,\alpha(e_k)$ are linearly independent.

Note that in general, the compatible connection in an abstract GKM graph is not unique. However, if the graph is GKM$_3$, then it is: given $v\in V(\Gamma)$ as well as $e,f\in E_v$, the GKM$_3$ condition implies that there can be only one edge $h\in E_{t(e)}$ such that $\alpha(e), \alpha(f)$ and $\alpha(h)$ are linearly dependent. As soon as we restrict to GKM$_3$ actions and graphs, we can therefore safely assume $\nabla$ to be fixed as part of the given structure.

The topological relevance of the GKM graph of a GKM action lies in the fact that it encodes the equivariant cohomology (and then, via Proposition \ref{prop:eqformalcomputable}, also the ordinary cohomology) algebra of the action, as originally shown in \cite[Theorem 7.2]{GKM} as a consequence of the Chang--Skjelbred Lemma \cite[Lemma 2.3]{ChangSkjelbred}: 
 
\begin{proposition}\label{prop:graphdeteqcohom}
The GKM graph of a GKM action determines the equivariant cohomology algebra. More precisely, ths inclusion $M^T\to M$ of the fixed point set induces an isomorphism of $R$-algebras
\[
		H^*_T(M) \cong \left\{\left.(f(u))_u\in \bigoplus_{u\in V(\Gamma)} R \,\,\,\right|  \,\, \begin{matrix} f(v) - f(w) \equiv 0 \mod \alpha(e) \\ \text{ for every edge $e$ from $v$ to $w$} \end{matrix} \right\}
	\]
\end{proposition} 
The right hand side of this equation is also called the equivariant graph cohomology of $\Gamma$, which is defined in the same way for an abstract GKM graph \cite{GZ2}.
 
Given an abstract GKM graph $(\Gamma,\alpha)$, with a fixed compatible connection $\nabla$, a vertex $v\in V(\Gamma)$ and edges $e,f\in E_v$, we obtain a unique sequence of edges by successively applying the connection: we put $e_1 = f$, $e_2 = e$, and then $e_k = \nabla_{e_{k-1}} e_{k-2}$ for all $k\geq 3$. This closed path is called a \emph{connection path} in $\Gamma$. As we observed above that for a GKM$_3$ graph the connection is unique, in the GKM$_3$ setting connection paths are independent of any additional choice. Moreover, while connection paths may self-intersect for a general GKM graph, this is not the case for GKM$_3$ graphs: all the edges $e_i$ in the connection path defined by $v$, $e$ and $f$  satisfy that $\alpha(e_i)$ vanishes on $\ker \alpha(e)\cap \ker \alpha(f)$, and by the GKM$_3$ condition at any vertex there can be at most two edges whose labels vanish on a given codimension two subspace. We will call the connection paths in a GKM$_3$ graph the \emph{$2$-faces} of $\Gamma$.

Geometrically, given a GKM$_3$ action on a manifold $M$, with GKM graph $\Gamma$, the $2$-skeleton 
\[
M_2 = \{p\in M\mid \dim T\cdot p\leq 2\}
\]
is a union of $4$-dimensional $T$-invariant submanifolds, which are GKM manifolds on their own. Their GKM graphs are exactly the $2$-faces of $\Gamma$.

\begin{definition} \label{defn:signedstructure}
Given an abstract GKM graph $(\Gamma,\alpha)$, a \emph{signed structure} on $\Gamma$ is a lift $\tilde{\alpha}:\tilde{E}(\Gamma) \to {\mathbb{Z}}^k$ of $\alpha$ such that
\begin{enumerate}
\item For all $e\in \tilde{E}(\Gamma)$ we have $\tilde{\alpha}(e) = -\tilde{\alpha}(\bar{e})$.
\item There exists a compatible connection satisfying that for all $v\in V(\Gamma)$ and $e,f\in E_v$ we have
\[
\tilde{\alpha}(\nabla_e f) \in \tilde{\alpha}(f) + {\mathbb{Z}} \tilde{\alpha}(e).
\]
\end{enumerate}
Given a signed structure on $\Gamma$, we will call $\Gamma$ a \emph{signed GKM graph}; sometimes, in case no signed structure is fixed, and we wish to emphasize this fact, we will call $\Gamma$ an \emph{unsigned GKM graph}.
\end{definition}

The original definition of an abstract GKM graph in \cite{GZ1, GZ2}  was that of a signed GKM graph, having in mind the geometric situation of a torus action with an invariant symplectic or almost complex structure, in which case the weights of the isotropy representations at the fixed points are well-defined elements in ${\mathfrak{t}}^*$. In our setting of actions on almost quaternionic manifolds the GKM graph will in general not admit a signed structure. However, the notion of a signed GKM graph is still relevant here, as many of the $2$-faces of the GKM graphs that occur will admit a signed structure, see Proposition \ref{prop:2facescomplex} below.

\section{Torus actions on almost quaternionic manifolds}\label{sec:torusonaqm}
In this section we will consider torus actions on almost quaternionic manifolds leaving invariant the structure. We will focus on the case of an action of GKM type, and find the natural combinatorial structure on the GKM graph induced by the almost quaternionic structure. To this end, we first prove a proposition about representations on quaternionic vector spaces that we will apply to the isotropy representations of our actions.

%recall real/complex/quaternionic reps. If weights are pw linearly independent, then necessarily complex.

%\textcolor{red}{compare Fuquan Fang, Positive Quaternionic K\"ahler Manifolds and Symmetry Rank, p.\ 7, using Kostant's Theorem....}

\begin{proposition}\label{prop:quatrep}
Consider a real representation of a torus $T$ on a finite-dimensional quaternionic vector space $V$. Let $Q\subset \End_\RR(V)$ be the $3$-dimensional subspace spanned by multiplication with $i, j$ and $k$. Assume that the induced $T$-representation on $\End_\RR(V)$ by conjugation leaves invariant $Q$. Then one of the following holds true:
\begin{enumerate}
\item The $T$-representation on $Q$ is trivial, i.e., the $T$-representation on $V$ is quaternionic.
\item $Q$ decomposes uniquely as $Q= Q_0\oplus Q'$ into the sum of a one-dimensional trivial $T$-module $Q_0$ and a two-dimensional weight space $Q'$. We fix any of the two elements $I\in Q_0$ that define a complex structure on $V$. The $T$-represen\-tation on $V$ is complex with respect to $I$, and $A\mapsto IA$ defines a $T$-invariant complex structure on $Q'$. Let $\lambda\in \mft^*$ be the weight of this complex representation.

Then the weights of the complex $T$-representation $V$ occur in pairs $\alpha, \lambda-\alpha$.
\end{enumerate}

\end{proposition}

\begin{proof}
We assume that (1) does not hold. Then $Q$ decomposes uniquely as $Q = Q_0\oplus Q'$ into the sum of a one-dimensional trivial $T$-module $Q_0$ and a two-dimensional weight space $Q'$. As we did not introduce a complex structure on $Q'$ yet, we have to consider the weight $\lambda$ of $Q'$ as an element in $\mft^*/\pm 1$ for the moment.

Any nonzero element in $Q$ is, up to rescaling, a complex structure on $V$. Hence, there are exactly two elements in the line $Q_0$ that define a complex structure on $V$; we fix one of them and call it $I$. Let us show that $I$ defines a complex structure on $Q'$ by composition: $Q'\to Q'; A\mapsto IA$.

For this, consider $\tilde{Q}:= Q\oplus \langle \id_V\rangle$. Also $\tilde{Q}$ is a $T$-module, and we have
\[
Q' = \{A\in \tilde{Q}\mid [X,[X,A]] = -\lambda(X)^2A \textrm{ for all }X\in \mft\}.
\]
Note that the expression $\lambda(X)^2$ is well-defined. Now observe that for $A\in Q'$, also $IA\in Q'$: we clearly have $IA\in \tilde{Q}$, and compute, using $[X,I]=0$ for all $X\in \mft$, that
\[
[X,[X,IA]] = X^2IA - 2XIAX + IAX^2 = I[X,[X,A]] = -\lambda(X)^2IA.
\]  
Observe that also as an endomorphism of $Q'$, $I$ squares to $-\id$. From now on, consider $\lambda\in \mft^*$ as a weight of the complex representation on $Q'=Q_\lambda$.

Let $V_\alpha$ be some weight space of $V$. In order to show that also $\lambda-\alpha$ is a weight of $V$, we claim that $Q_\lambda \cdot V_\alpha = V_{\lambda -\alpha}$.

To see this, fix $v\in V_\alpha$ nonzero. Choose $J\in Q_\lambda$; then $Q_\lambda$ is spanned by $J$ and $IJ$. Note that because $IJ$ is again an element of $Q$, $IJ = -JI$ and the elements $I$, $J$, and $IJ$ form a standard basis of the imaginary quaternions.

Using this, for $X\in \mft$, we compute $X\cdot Jv = [X,J] v + J(X\cdot v) = \lambda(X)IJ v + J(\alpha(X)Iv) = (\lambda(X)-\alpha(X))I(Jv)$.
\end{proof}
\begin{remark}
If one replaces $I$ by $-I$ in case (2), then $\lambda$ as well as all the weights of the $T$-representation $V$ change their sign. This is consistent with the description of the weights as pairs $\alpha$, $\lambda-\alpha$, as $-(\lambda-\alpha) = -\lambda - (-\alpha)$.
\end{remark}
\begin{remark}
If one restricts the $T$-representation on $V$ to $\ker \lambda$, one obtains, by definition of $\lambda$, a quaternionic (i.e., ${\mathbb{H}}$-linear) representation on $V$.
\end{remark}

Consider an action of a torus $T$ on an almost quaternionic manifold $M$ by automorphisms, and denote the three-dimensional subbundle defining the almost quaternionic structure by $Q\subset \End(TM)$. Then, for any $p\in M$, the tangent space $T_pM$ becomes a quaternionic vector space, uniquely determined up to an identification of $Q_p$ with the imaginary quaternions. 

Consider the isotropy representation of $T_p$ on $T_p M$, as well as its induced representation on $\End T_pM$. By definition of an almost quaternionic automorphism, this representation leaves invariant $Q_p\subset \End_\RR T_pM$. Restricting to the identity component of $T_p$, we are hence in the situation of Proposition \ref{prop:quatrep}.

Assuming that the $T$-action is of GKM type, for any $p\in M^T$ the $T$-isotropy representation has pairwise linearly independent weights, and in particular cannot be quaternionic. Hence, we are always in case (2) of Proposition \ref{prop:quatrep}.

\begin{proposition}\label{prop:aqmfixedsubmfd}
Consider an action of a torus $T$ on a simply-connected almost quaternionic manifold $M$ by automorphisms, and $H\subset T$ a subtorus. Let $N$ be a connected component of $M^H$. Then:
\begin{enumerate}
\item If the $H$-action on the vector bundle $Q|_N$ is trivial, then $N$ is an almost quaternionic submanifold.
\item If the $H$-action on the vector bundle $Q|_N$ is not trivial, then there is a (up to sign unique) section of $Q|_N$ which defines a $T$-invariant almost complex structure on $N$.
\end{enumerate}
\end{proposition}
\begin{proof}
Consider the $H$-action on the vector bundle $Q|_{N}$. As the $H$-action on $N$ is trivial, all $H$-representations on the fibers are isomorphic. If the $H$-action on $Q|_N$ is trivial, then each element in $Q|_N$ leaves invariant $TN$. By definition, this means that $N$ is an almost quaternionic submanifold. 

If the $H$-action on $Q|_N$ is not trivial, then $Q|_N$ decomposes as the sum of a line bundle pointwise fixed by $H$ and a two-dimensional invariant subbundle without fixed vectors. As $M$ is simply-connected, the vector bundle $Q$ on $M$ is necessarily orientable. Hence, the restriction of $Q$ to $N$ is orientable as well.  Moreover, as the two-dimensional invariant subbundle has no fixed vectors, it is also orientable (choose $X\in {\mathfrak{t}}$ acting nontrivially on the fibers; then an orientation is given by the fiberwise bases $(v,X\cdot v)$, with $v$ any vector in a fiber of the subbundle). Hence, its complement, the line subbundle of $H$-fixed vectors, is orientable, thus trivial; an appropriately normalized section of it defines an almost complex structure on $N$, which is automatically $T$-invariant.
\end{proof}

\begin{definition}\label{defn:quatstrgraph}
Let $(\Gamma,\alpha)$ be an abstract (unsigned) GKM graph. Then a \emph{quaternionic structure} on $\Gamma$ consists of the choice,  for any vertex $v$, of an element $\lambda(v)\in \mft^*/\pm 1$ (called \emph{quaternionic weight}) as well as of a grouping of the edges at $v$ into pairs (called \emph{quaternionic pairs of edges}) such that the following two conditions are satisfied.
\begin{enumerate}
\item For any quaternionic pair of edges $e,f$ at a vertex $v$, there are lifts $\tilde{\alpha}(e)$ and $\tilde{\alpha}(f)$ of the weights $\alpha(e)$ and $\alpha(f)$ such that $\pm (\tilde{\alpha}(e)+\tilde{\alpha}(f))=\lambda(v)$.
\end{enumerate}
Having fixed the quaternionic pairs of edges $e_i$, $f_i$, $i=1,\ldots,n$, at a vertex $v$, we consider the weights $\alpha(e_i)$ and $\alpha(f_i)$, as well as the quaternionic weight $\lambda(v)$. By (1), if we choose a lift of any one of these weights, there are uniquely determined lifts of all other weights $\tilde{\alpha}(e_i)$, $\tilde{\alpha}(f_i)$ and $\tilde{\lambda}(v)$ such that $\tilde{\alpha}(e_i) + \tilde{\alpha}(f_i) = \tilde{\lambda}(v)$ for all $i$. We will call any subset of these lifts \emph{quaternionically compatible}. Clearly, if a set of lifts is quaternionically compatible, then they stay quaternionically compatible when we replace them all by their negatives.
\begin{enumerate}
\item[(2)] For any edge $e$ and lifts $\tilde{\alpha}(e)$ and $\tilde{\lambda}(i(e))$ that are quaternionically compatible, there is a lift $\tilde{\lambda}(t(e))$ which satisfies  $\tilde{\lambda}(t(e))-\tilde{\lambda}(i(e))\equiv 0 \mod \alpha(e)$ and which is quaternionically compatible with $-\tilde{\alpha}(e)$. 

Further, there shall exist a compatible connection $\nabla$ which respects quaternionic pairs of edges, and which satisfies additionally that for any edge $f$ at $i(e)$, with lift $\tilde{\alpha}(f)$ chosen quaternionically compatibly with $\tilde{\alpha}(e)$ and $\tilde{\lambda}(i(e))$, the lift $\tilde{\alpha}(\nabla_e f)$ with  $\tilde{\alpha}(\nabla_e f)- \tilde{\alpha}(f)\equiv 0 \mod \alpha(e)$ is quaternionically compatible with $\tilde{\lambda}(t(e))$ and $-\tilde{\alpha}(e)$. (See Figure \ref{fig:quaternioniccompatibility}. In this and the following figures, the quaternionic weights are depicted in a box.)
 \end{enumerate}
\end{definition}

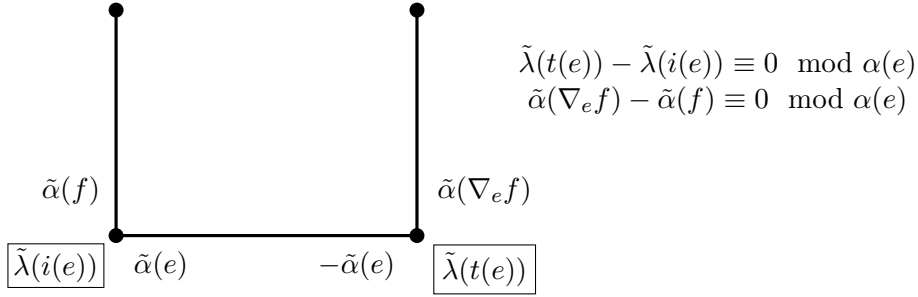
\begin{figure}[h]
		\begin{tikzpicture}

			\draw [very thick] (0,0) -- ++ (4,0) -- ++ (0,3);
			\draw [very thick] (0,0) -- ++ (0,3);
			\node at (0,0)[circle,fill,inner sep=2pt] {};
			\node at (0,3)[circle,fill,inner sep=2pt]{};
			\node at (4,3)[circle,fill,inner sep=2pt]{};
			\node at (4,0)[circle,fill,inner sep=2pt]{};
			\node at (-.8,-.4)[rectangle, draw, inner sep =2pt]{$\tilde{\lambda}(i(e))$};
			\node at (.6,-.4) {$\tilde{\alpha}(e)$};
			\node at (4.9,-.45) [rectangle, draw, inner sep = 3pt]{$\tilde{\lambda}(t(e))$};
			\node at (3.2,-.4) {$-\tilde{\alpha}(e)$};

			\node at (-.6,.6) {$\tilde{\alpha}(f)$};

			\node at (4.9,.6) {$\tilde{\alpha}(\nabla_e f)$};
				
				\node at (8,1.8) {$\tilde{\alpha}(\nabla_ef)-\tilde{\alpha}(f) \equiv 0 \mod \alpha(e)$};
				\node at (8,2.3) {$\tilde{\lambda}(t(e))-\tilde{\lambda}(i(e)) \equiv 0 \mod \alpha(e)$};

		\end{tikzpicture}
		\caption{Quaternionically compatible choices of lifts}\label{fig:quaternioniccompatibility}
		\end{figure}

\begin{proposition}\label{prop:gammahasquat}
Consider a GKM action of a torus $T$ on an almost quaternionic manifold by automorphisms. Then the GKM graph of $M$ admits a quaternionic structure.
\end{proposition}
\begin{proof}
For a vertex $v$, we let $\lambda(v)$ be the weight of $Q_v$ as provided by Proposition \ref{prop:quatrep}. This proposition also gives us the grouping of the edges into quaternionic pairs such that condition (1) is satisfied.

For an edge $e$ we may apply Proposition \ref{prop:aqmfixedsubmfd} to the corresponding $2$-sphere $N$. As it is $2$-dimensional, it cannot be an almost quaternionic submanifold, so we are necessarily in case (2) of that proposition, and find an (up to sign unique) section $J$ of $Q$ along the $2$-sphere, which defines an invariant almost complex structure on $TM|_N$, and in particular on $N$. The standard argument for the existence of a connection, see \cite[Proposition 2.3 and Remark 2.6]{GW}, works as follows: we define $T_N$ to be the principal isotropy group of the $T$-action on $N$, and decompose $TM|_N$, considered as a complex $T_N$-vector bundle, equivariantly into $T_N$-invariant complex line bundles $L_i$. In the fibers over the two $T$-fixed points in $N$, every line bundle corresponds uniquely to an edge of the GKM graph. Then the connection $\nabla$ is defined by following the respective line bundle along $N$. 

In our setting, we may choose $L_1 = TN$ and $L_{2n} = \tilde{Q}\cdot L_{2n-1}$ for all $n$, where $\tilde{Q}=\{I\in Q|_N\mid JI=-IJ\}$, so that $L_{2n-1}\oplus L_{2n}$ is a $Q$-invariant subbundle of $TM|_N$, which implies that we may choose the connection $\nabla$ to respect quaternionic pairs.
\end{proof}

\begin{remark}
A different point of view on the quaternionic structure on the GKM graph of a GKM action by automorphisms on an almost quaternionic manifold is given by the induced action on the twistor space. Consider the twistor space $Z\to M$ over $(M,Q)$ with its natural almost complex structure $J^Q$ (see Section \ref{sec:quaternionicgeometries}), together with its induced $T$-action. We observe that this action preserves $J^Q$: if $\nabla$ is an Oproiu connection, $\tilde{\nabla}:= dt\circ \nabla \circ dt^{-1}$ is again an Oproiu connection for every $t\in T$. Indeed the torsion of $\tilde{\nabla}$ is
\[
 T^{\tilde{\nabla}}(X,Y)=dt (T^Q(dt^{-1}X,dt^{-1}Y)),
\]
which, looking at the definition \eqref{oproiutorsion} of $T^Q$, gives again the same expression as in \eqref{oproiutorsion} but with respect to a different orthonormal base of $Q$. Then, since the definition of $T^Q$ is independent of the choice of local base of $Q$, we have that $T^{\tilde{\nabla}}= T^Q$.

Observe moreover that the induced action on $Z$ is again of GKM type (but it can never be of type GKM$_3$, as the weights $\alpha, \lambda-\alpha$ and $\lambda$ are always linearly dependent). Its GKM graph is naturally a GKM fiber bundle (see \cite{GSZ}) over the GKM graph $\Gamma$ of $M$. The fiber over any vertex $p$ of $\Gamma$ is a line whose label is given by the quaternionic weight of $p$. 

For instance, the twistor space of the quaternion-K\"ahler manifold $S^4=\HH P^1$ is $\CC P^3$, see \cite[Section 5.2]{Salamon}. Considering the standard $T^2$-action on $\HH P^1$, as well as the induced action on the twistor space, the associated GKM fibration is depicted in Figure \ref{fig:twistorS4}.

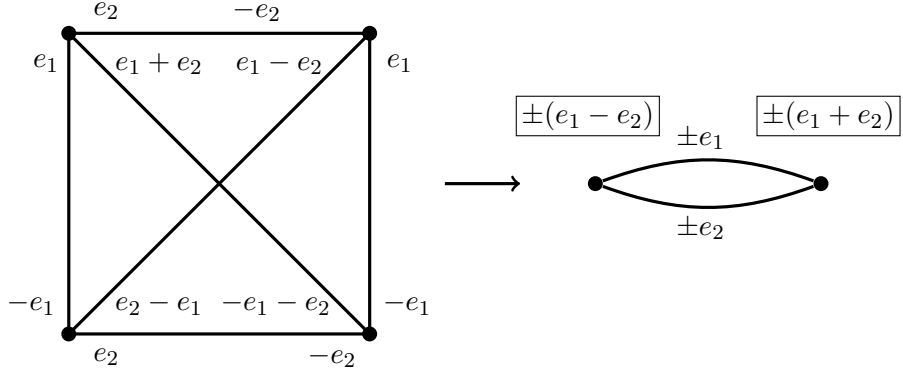
\begin{figure}[h]
		\begin{tikzpicture}

			\draw [very thick] (0,0) -- (4,0) -- (4,-4) -- (0,-4) -- (0,0) -- (4,-4);
			\draw [very thick] (4,0) -- (0,-4);	
			\node at (0,0)[circle,fill,inner sep=2pt] {};
			\node at (4,-4)[circle,fill,inner sep=2pt]{};
			\node at (0,-4)[circle,fill,inner sep=2pt]{};
			\node at (4,0)[circle,fill,inner sep=2pt]{};
			\node at (.5,.3) {$e_2$};
			\node at (2.5,.3) {$-e_2$};			
			\node at (-.3,-.4) {$e_1$};

			\node at (-.5,-3.6) {$-e_1$};

			\node at (.5,-4.3) {$e_2$};
			
			\node at (3.5,-4.3) {$-e_2$};

			\node at (4.5,-3.6) {$-e_1$};
				
			\node at (4.4,-.4) {$e_1$};
			\node at (1.2,-.4) {$e_1+e_2$};
			\node at (2.8,-.4) {$e_1-e_2$};
			\node at (1.2,-3.6) {$e_2-e_1$};
			\node at (2.75,-3.6) {$-e_1-e_2$};
\draw[->, very thick] (5,-2) -- ++(1,0);

\node (a) at (7,-2)[circle,fill,inner sep=2pt] {};
			\node (b) at (10,-2)[circle,fill,inner sep=2pt]{};
			\node at (8.4,-1.4) {$\pm e_1$};
			\node at (8.4,-2.6) {$\pm e_2$};

			\node at (6.9,-1.1)[rectangle, draw, inner sep =2pt]{$\pm (e_1-e_2)$};
			\node at (10.1,-1.1)[rectangle, draw, inner sep =2pt]{$\pm (e_1+e_2)$};
			\draw [very thick](a) to[in=160, out=20] (b);
			\draw [very thick](a) to[in=200, out=-20] (b);
		\end{tikzpicture}
		\caption{The GKM twistor fibration of $\HH P^1$}\label{fig:twistorS4}
		\end{figure}

\end{remark}

\begin{example}\label{ex:4-mfdsigned}
Consider any isometric GKM $T^2$-action on a connected, compact, oriented Riemannian $4$-manifold $M$ with vanishing odd rational cohomology. By Example \ref{ex:4diminvalmostquat}, this action admits an invariant almost quaternionic structure $Q_+$, so that its GKM graph admits a quaternionic structure.

Further, in case $M$ admits an invariant almost complex structure $J$, the GKM graph admits a signed structure. If $J$ is compatible with the Riemannian metric and orientation, we argued in Example \ref{ex:4diminvalmostquat} that $J$ is automatically a section of $Q_+$. In this case, the quaternionic weight at a vertex $v$ is just given by the sum of the weights of the edges at this vertex. 

For instance, see  Figure \ref{fig:actionsonKaehlerCP2} for the GKM graph of a standard $T^2$-action on the K\"ahler manifold $\CC P^2$. 
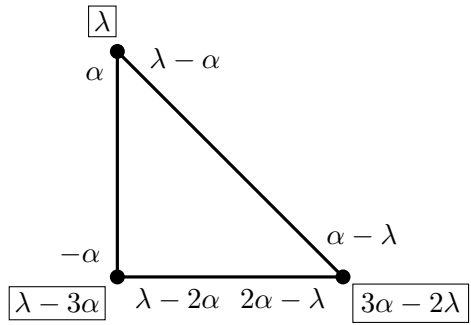
\begin{figure}[h]
		\begin{tikzpicture}

			\draw [very thick] (0,0) -- (3,-3) -- (0,-3) -- (0,0);	
			\node at (0,0)[circle,fill,inner sep=2pt] {};
			\node at (3,-3)[circle,fill,inner sep=2pt]{};
			\node at (0,-3)[circle,fill,inner sep=2pt]{};
			\node at (.9,-.1) {$\lambda-\alpha$};
			\node at (-.2,.4)[rectangle, draw, inner sep =2pt]{$\lambda$};
			\node at (-.3,-.3) {$\alpha$};

			\node at (-.5,-2.7) {$-\alpha$};
			\node at (-.8,-3.35) [rectangle, draw, inner sep=2pt]{$\lambda-3\alpha$};
			\node at (.8,-3.3) {$\lambda-2\alpha$};
			
			\node at (2.2,-3.3) {$2\alpha-\lambda$};
			\node at (3.9,-3.35) [rectangle, draw, inner sep = 3pt]{$3\alpha-2\lambda$};
			\node at (3.25,-2.45) {$\alpha-\lambda$};

		\end{tikzpicture}
		\caption{Quaternionic structure on the GKM graph of the K\"ahler manifold $\CC P^2$}\label{fig:actionsonKaehlerCP2}
		\end{figure}
		
Note that the quaternionic structure on the GKM graph of the standard $T^2$-action on the quaternion-K\"ahler manifold $\CC P^2$ looks different, see Proposition \ref{prop:actionsonG2Cn} and Figure \ref{fig:actionsonCP2withsign} below. Indeed, the Riemannian metrics of both structures coincide, but the orientation induced by the quaternion-K\"ahler structure on $\CC P^2$ is the opposite of the orientation induced by the K\"ahler structure. In the notation of Example \ref{ex:4diminvalmostquat}, if we fix the orientation compatible with the quaternion-K\"ahler structure, the complex structure $J$ gives a section of $Q_-$.
\end{example}

\begin{remark}
Note that there is no reason to expect the GKM graph of a GKM action on a general almost quaternionic manifold to admit a signed structure. For instance, compact positive quaternion-K\"ahler manifolds other than the Grassmannian $\Gr_2(\CC^{n+2})$ do not admit an almost complex structure, see \cite[Theorem 1.1]{GMS}. %most real grassmannians as well as $\mathbb{H}P^n$ do not admit any complex structure, see \cite{Massey}, \cite{Sankaran}, \cite{Tang}, and also .} \marginpar{cite also Hirzebruch for $HP^n$, $n>1$, and Ehresmann and Hopf for $HP^1$.?} 
\end{remark}

\begin{example}\label{ex:quatbiangle}
Every abstract GKM graph of the form of a biangle admits a quaternionic structure. All possibilities, including quaternionically compatible lifts, up to reversing all signs at a vertex, are shown in Figure \ref{fig:quatbiangles}. 
\begin{figure}[h]
	\begin{center}
		\begin{tikzpicture}

			\node (a) at (0,0)[circle,fill,inner sep=2pt] {};
			\node (b) at (3,0)[circle,fill,inner sep=2pt]{};
			\node at (.1,.4) {$\lambda-\alpha$};
			\node at (-.6,0)[rectangle, draw, inner sep =2pt]{$\lambda$};
			\node at (.3,-.4) {$\alpha$};
			\node at (3.8,0)[rectangle, draw, inner sep = 2pt]{$\lambda-2\alpha$};
			\node at (2.7,.4){$\lambda-\alpha$};
			\node at (2.6,-.4){$-\alpha$};

			\draw [very thick](a) to[in=160, out=20] (b);
			\draw [very thick](a) to[in=200, out=-20] (b);

		\end{tikzpicture}
	\end{center}
	\caption{Quaternionically compatible sign choices on a biangle}\label{fig:quatbiangles}
	\end{figure}
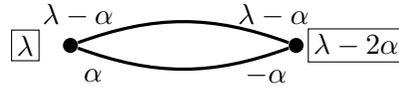 
	
	Indeed, choosing $\alpha$ and $\lambda$ in such a way that the three labels on the left in Figure \ref{fig:quatbiangles} are as depicted, and fixing $-\alpha$ as a label on the lower right, the label $\alpha-\lambda$ on the upper right would not give a quaternionic structure.
\end{example}

\begin{example}\label{ex:typesoftriangles} Every abstract GKM graph of the form of a triangle admits a quaternionic structure. There are two conceptually different types of labelings: those coming from an action on the Kähler manifold $\CC P^2$, and those that do not. Figure \ref{fig:actionsonKaehlerCP2} described the first, while the latter are given in Figure \ref{fig:actionsonCP2withsign}. Below in Section \ref{sec:thewolfspaces} we will observe that the latter describe actions on the quaternion-Kähler manifold $\CC P^2$. \begin{figure}[h]
		\begin{tikzpicture}

			\draw [very thick] (0,0) -- (3,-3) -- (0,-3) -- (0,0);	
			\node at (0,0)[circle,fill,inner sep=2pt] {};
			\node at (3,-3)[circle,fill,inner sep=2pt]{};
			\node at (0,-3)[circle,fill,inner sep=2pt]{};
			\node at (.9,-.1) {$\lambda-\alpha$};
			\node at (-.2,.4)[rectangle, draw, inner sep =2pt]{$\lambda$};
			\node at (-.3,-.3) {$\alpha$};

			\node at (-.5,-2.7) {$-\alpha$};
			\node at (-.6,-3.4) [rectangle, draw, inner sep=2pt]{$\lambda-\alpha$};
			\node at (.4,-3.3) {$\lambda$};
			
			\node at (2.6,-3.3) {$\lambda$};
			\node at (3.3,-3.35) [rectangle, draw, inner sep = 3pt]{$\alpha$};
			\node at (3.25,-2.45) {$\alpha-\lambda$};

		\end{tikzpicture}
		\caption{Quaternionically compatible sign choices on a noncomplex triangle}\label{fig:actionsonCP2withsign}
		\end{figure}
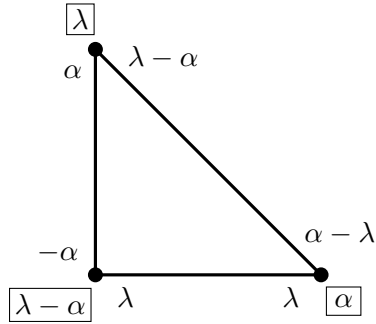

These two figures describe all possibilities, including quaternionically compatible lifts, up to reversing all signs at some vertices. To see this, let $\alpha$ and $\lambda$ be such that the labeling at the upper vertex is given as in Figure \ref{fig:actionsonKaehlerCP2} respectively \ref{fig:actionsonCP2withsign}. Transporting the diagonal edge along the edge labelled $\pm \alpha$ shows that the bottom edge is labelled $\pm(\lambda + c\alpha)$ for some integer $c$. Transporting the bottom edge along the diagonal edge shows that $c=0$ or $c=-2$, corresponding to Figure \ref{fig:actionsonCP2withsign} respectively \ref{fig:actionsonKaehlerCP2}.

We will call an abstract GKM graph as in Figure \ref{fig:actionsonKaehlerCP2} a \emph{complex triangle}, and one as  in Figure \ref{fig:actionsonCP2withsign} a \emph{noncomplex triangle}.
\end{example}

\begin{remark}
If an abstract GKM graph with quaternionic structure is at least $4$-valent, then it can never be of type GKM$_4$. In fact, in this case at any vertex we have at least two pairs of edges with labels $\alpha,\lambda-\alpha$ and $\beta, \lambda-\beta$. These four labels are linearly dependent.
\end{remark}

\begin{remark}\label{rem:signcompatibilitiesintheunsignedgraph}
Consider an abstract GKM$_3$ graph $\Gamma$ with a quaternionic structure, and fix a vertex $v$. We claim that the quaternionic pairs of edges and the quaternionically compatible lifts are uniquely determined by the unsigned weights $\alpha(e)$ and $\lambda(v)$ alone. More precisely: fix any lift $\tilde{\lambda}(v)$ of $\lambda(v)$. Now consider any edge $e$ at $v$, with label $ \alpha(e)$, lifted arbitrarily to $\tilde{\alpha}(e)$. There are two possibilities for the label of an edge $f$ at $v$ that forms a quaternionic pair with $e$: it may have label $\pm(\tilde{\lambda}(v)-\tilde{\alpha}(e))$, because $\tilde{\alpha}(e) + (\tilde{\lambda}(v)-\tilde{\alpha}(e))=\tilde{\lambda}(v)$, or $\pm (\tilde{\lambda}(v) + \tilde{\alpha}(e))$, because $-\tilde{\alpha}(e) + (\tilde{\lambda}(v) + \tilde{\alpha}(e)) = \tilde{\lambda}(v)$. But by the GKM$_3$ condition, there cannot exist edges with labels $\pm \tilde{\alpha}(e)$, $\pm (\tilde{\lambda}(v)-\tilde{\alpha}(e))$ and $\pm (\tilde{\lambda}(v) + \tilde{\alpha}(e))$ at the same time, so it is clear from the unsigned weights which case occurs. Note that the outcome would be the same if we had chosen the other lifts of $\lambda(v)$ and $\alpha(e)$.
\end{remark}

Consider an abstract GKM graph with a quaternionic structure. If we assume that it is of type GKM$_3$, then any two edges at a vertex $v$ define a unique $2$-face, as explained in Section \ref{sec:gkmactions}. As the connection preserves quaternionic pairs of edges, either in every vertex of a $2$-face the two adjacent edges form a quaternionic pair or in none. In the first case we speak about a \emph{quaternionic $2$-face}. The next proposition justifies why we call the other $2$-faces \emph{complex $2$-faces}.

\begin{proposition}\label{prop:2facescomplex}
Let $\Gamma$ be an abstract GKM$_3$ graph with a quaternionic structure. Then any complex $2$-face admits a signed structure which is compatible with the quaternionic structure on $\Gamma$ in the sense that at every vertex of the $2$-face, the signed weights at the adjacent edges are quaternionically compatible.
\end{proposition}
\begin{proof}
Choose any vertex $v$ of the complex $2$-face, and call the edges of the $2$-face $e_1,\ldots,e_n$, with $i(e_1)=t(e_n)=v$ and $t(e_j)=i(e_{j+1})$. We write $\alpha$ and $\beta$ for quaternionically compatible lifts of $\alpha(e_1)$ and $\alpha(e_n)$, and denote the compatibly lifted quaternionic weight at $v$ by $\lambda$. 

Then, by definition of a quaternionic structure, the lift of the quaternionic weight at $t(e_1)$ which is quaternionically compatible with $-\alpha$ is of the form $\lambda + c_1\alpha$. Let $\tilde{\alpha}(e_2)$ be the lift of $\alpha(e_2)$ quaternionically compatible  with these elements; then $\lambda + c_1\alpha + c_2\tilde{\alpha}(e_2)$ is a lift of the quaternionic weight at $t(e_2)$ which is quaternionically compatible with $-\tilde{\alpha}(e_2)$. Note that all the labels of edges of the $2$-face are linear combinations of $\alpha$ and $\beta$. We continue this process,  transporting the quaternionic weight around the connection path, at the same time introducing a signed structure on the $2$-face compatible with the quaternionic structure. Closing the connection path, in the last step we obtain a lift of the quaternionic weight at $v$, again of the form $\lambda + a\alpha + b\beta$, which is quaternionically compatible with $-\tilde{\alpha}(e_n)$. If $-\tilde{\alpha}(e_n) = \beta$, then we have obtaind a signed structure compatible with the quaternionic structure. But if $-\tilde{\alpha}(e_n) = -\beta$, then necessarily $-\lambda = \lambda + a\alpha + b\beta$, which implies that $\lambda$, $\alpha$ and $\beta$ are linearly dependent. But this is a contradiction to the GKM$_3$ condition, as from the vertex $v$ there emerge edges labelled $\alpha$, $\beta$ and $\lambda-\alpha$.
\end{proof}

\begin{remark}
It is possible that also the quaternionic $2$-faces admit a compatible signed structure, see Example \ref{ex:4-mfdsigned}, but in general they do not.
\end{remark}

\section{Torus actions on quaternion-K\"ahler manifolds}

In the positive quaternion-K\"ahler setting, we can refine Proposition \ref{prop:aqmfixedsubmfd}. Note that the assumption of simple connectivity is satisfied for positive quaternion-K\"ahler manifolds by \cite[Theorem 6.6]{Salamon}.
\begin{proposition}\label{prop:structureonfixedsubmfd}
Consider an action of a torus $T$ on a positive quaternion-K\"ahler manifold $M$ by quaternion-K\"ahler automorphisms, and $H\subset T$ a subtorus. Let $N$ be a connected component of $M^H$. Then $N$ is either
\begin{enumerate}
\item a quaternion-K\"ahler submanifold, or
\item a K\"ahler-Einstein submanifold, with $T$-invariant complex structure given by a section of $Q|_N$.
\end{enumerate}
\end{proposition}
\begin{proof}
 In case $H$ acts trivially on $Q|_N$, we showed in Proposition \ref{prop:aqmfixedsubmfd} that $N$ is an almost quaternionic submanifold. As $N$ is automatically totally geodesic, the almost quaternionic structure of $N$ is preserved by the Levi-Civita connection. Hence, in dimension bigger than $4$, $N$ is quaternion-Kähler by definition. In dimension $4$, $N$ is Einstein (since $M$ is Einstein, see \cite{Alekseevsky} or \cite[Theorem 3.1]{Salamon}) and the (complex) twistor bundle of $M$ restricts to the twistor bundle of $N$, which is again complex, and thus $N$ is a quaternion-K\"ahler manifold, see \cite[Theorem 13.46]{Besse}. 

In case $H$ acts non-trivially on $Q|_N$, we constructed a $T$-invariant almost complex structure on $N$ as a section of $Q|_N$. Therefore, $N$ is an almost Hermitian submanifold in the sense of \cite{AM} and hence, by \cite[Corollary 1.10]{AM} it is either a quaternion-Kähler or a Kähler submanifold. But by \cite[Theorem 3.9]{AMP1} (see also Footnote \ref{footnote1} for dimension $4$) the $T$-invariant almost complex structure cannot exist in the quaternion-Kähler case, so it is a K\"ahler submanifold of $M$. It is clearly Einstein as it is a  totally geodesic submanifold of $M$, which is Einstein. 
\end{proof}

\begin{remark} \label{rem:compatibilitycomplexstructures}
Consider the second case in Proposition \ref{prop:structureonfixedsubmfd}, i.e., a connected component $N$ of $M^H$ which is K\"ahler-Einstein. The complex structure is given by a section of $Q|_N$ that is pointwise fixed by $H$, and up to a global sign there is only one such section of $Q|_N$. On the other hand, assuming that the action is GKM, in any $T$-fixed point $p\in N$ the $T$-representation on $Q_p$ has a unique fixed line, as in case (2) of Proposition \ref{prop:quatrep}. Thus, we may choose the complex structure in $Q_p$ on $T_pM$ to extend the complex structure on $T_pN$ coming from the K\"ahler-Einstein structure. In the context of abstract GKM$_3$ graphs the corresponding statement was shown in Proposition \ref{prop:2facescomplex}.
\end{remark}

From now on, assume that the $T$-action on the positive quaternion-Kähler manifold $M$ is of type GKM$_3$, so that any two edges at a vertex $v$ of its GKM graph define a $2$-face, corresponding to a $4$-dimensional invariant submanifold. By Proposition \ref{prop:structureonfixedsubmfd}, this submanifold is either a quaternion-K\"ahler or a K\"ahler-Einstein submanifold, corresponding to the distinction between quaternionic and complex $2$-faces, see Section \ref{sec:torusonaqm}.

\begin{proposition}\label{prop:4dimsubmanifolds}
The four-dimensional submanifold corresponding to a $2$-dimensonal face is either 
\begin{enumerate}
\item a quaternion-K\"ahler submanifold. In this case it is either $\CC P^2$ or $\HH P^1$.
\item a toric K\"ahler-Einstein manifold. As a toric symplectic manifold, it is determined by its momentum image, which is, up to rescaling, one of the following three reflexive polytopes:
 
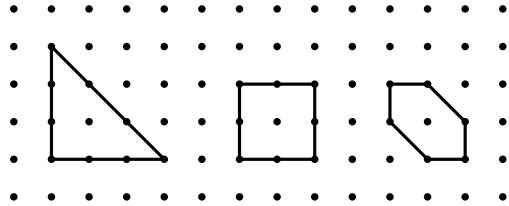
\begin{center}
\begin{tikzpicture}

\foreach \x in {-9,...,4}
    \foreach \y in {1,...,6}
    {
    %\draw (\x/2,\y/2) circle (2pt);
      \node at (\x/2,\y/2)[circle,fill,inner sep=1pt]{};
    }

\draw[very thick] (-8/2,5/2) -- ++(3/2,-3/2) -- ++(-3/2,0) -- ++(0,3/2) ;

\draw[very thick] (-3/2,4/2) -- ++(2/2,0) -- ++(0,-2/2) -- ++(-2/2,0)  -- ++(0,2/2);

\draw[very thick] (1/2,4/2) -- ++ (1/2,0) -- ++(1/2,-1/2) -- ++ (0,-1/2) --++ (-1/2,0) --++ (-1/2,1/2) --++ (0,1/2);

\end{tikzpicture}
\end{center}
\end{enumerate}
\end{proposition}
\begin{proof}
We showed in Proposition \ref{prop:structureonfixedsubmfd} that the four-dimensional submanifold is either quaternionic-K\"ahler or toric K\"ahler-Einstein. The first statement then follows directly from the classification of four-dimensional quaternion-K\"ahler manifolds due to \cite{Hitchin} or \cite{FriedrichKurke}.

In the other case it is a toric K\"ahler-Einstein $4$-fold. On K\"ahler manifolds the Ricci form represents the first Chern class, which implies that the first Chern class of a K\"ahler-Einstein manifold is a multiple of the class of the symplectic form. By definition, this means that the manifold is monotone symplectic. By \cite[Lemma 5.2]{1224} this multiple is positive if there is a Hamiltonian circle action, so we may rescale the symplectic form such that the first Chern class equals $[\omega]$. Then the momentum polytope is reflexive by \cite[Proposition 3.10]{1224}.

There are only 16 reflexive polytopes in dimension $2$, see \cite[Figure 2]{PoonenRodriguez}. Inspecting these pictures, one sees that only five of those satisfy the Delzant condition (cf., e.g., \cite[Section 3.2]{ChartonSabatiniSepe}). These are the three depicted in the statement of the proposition, as well as the following two:

 
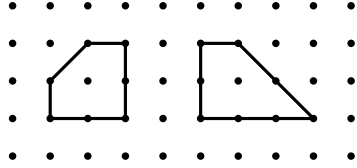
\begin{center}
\begin{tikzpicture}

\foreach \x in {-9,...,0}
    \foreach \y in {1,...,5}
    {
    %\draw (\x/2,\y/2) circle (2pt);
      \node at (\x/2,\y/2)[circle,fill,inner sep=1pt]{};
    }

\draw[very thick] (-7/2,4/2) -- ++(1/2,0) -- ++(0,-2/2) -- ++(-2/2,0) --++ (0,1/2) --++ (1/2,1/2) ;

\draw[very thick] (-4/2,4/2) -- ++(1/2,0) -- ++(2/2,-2/2) -- ++(-3/2,0)  -- ++(0,2/2);

\end{tikzpicture}
\end{center}
These latter two correspond  to $\CC P^2$, blown up at $1$ or $2$ fixed points. However, it was shown in \cite{Tian} that a compact complex surface can only admit a K\"ahler-Einstein metric with positive Einstein constant if it is $\CC P^2$, $\CC P^1\times \CC P^1$, or $\CC P^2$ blown up at $3$ to $8$ points in general position. Therefore, these two cases are impossible.
\end{proof}

This proposition shows that the quaternionic $2$-faces of the GKM graph $\Gamma$ of a positive quaternion-K\"ahler GKM$_3$ manifold are biangles or (noncomplex) triangles, and the complex $2$-faces are triangles, squares or $6$-gons. Now we will show that complex $6$-gons do not exist:

\begin{lemma}\label{lem:no6gon}
There do not exist complex $6$-gons in $\Gamma$.
\end{lemma}
\begin{proof}
Assume that there exists a complex $6$-gon in $\Gamma$. Recall from Remark \ref{rem:compatibilitycomplexstructures} (or from Proposition  \ref{prop:2facescomplex}) that it has a compatible signed structure. By Proposition \ref{prop:4dimsubmanifolds}, it has labels as in the following picture.
	
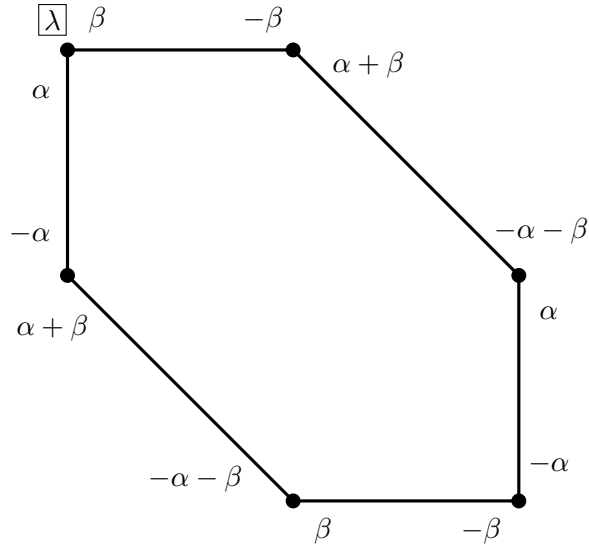
\begin{center}
		\begin{tikzpicture}

			\draw [very thick] (0,0) -- (3,0) -- (6,-3) -- (6,-6) -- (3,-6) -- (0,-3) -- (0,0);
			
			\node at (0,0)[circle,fill,inner sep=2pt] {};
			\node at (3,0)[circle,fill,inner sep=2pt]{};
			\node at (6,-3)[circle,fill,inner sep=2pt]{};
			\node at (6,-6)[circle,fill,inner sep=2pt] {};
			\node at (3,-6)[circle,fill,inner sep=2pt]{};
			\node at (0,-3)[circle,fill,inner sep=2pt]{};

			\node at (.4,.4) {$\beta$};
			\node at (-.2,.4)[rectangle, draw, inner sep =2pt]{$\lambda$};
			\node at (-.35,-.55) {$\alpha$};

			\node at (-.5,-2.45) {$-\alpha$};
			\node at (-.2,-3.7) {$\alpha + \beta$};
			
			\node at (1.7,-5.7) {$-\alpha-\beta$};
			\node at (2.6,.4) {$-\beta$};
			\node at (4,-.2) {$\alpha + \beta$};
			\node at (6.3,-2.4) {$-\alpha - \beta$};
			\node at (6.4,-3.5) {$\alpha$};
			\node at (6.4,-5.5) {$-\alpha$};
			\node at (3.4,-6.4) {$\beta$};
			\node at (5.5,-6.4) {$-\beta$};

		\end{tikzpicture}
	\end{center}

	Note that this figure is symmetric in the sense that for any two edges $e$ and $f$ at a vertex, the label of the successive edge in the connection path through $e$ and $f$ is exactly the sum $\alpha(e) + \alpha(f)$.

We enumerate the vertices clockwise $p_1,\ldots,p_6$, starting in the upper left corner. At $p_1$ there exists a further edge with label $\lambda-\alpha$. Consider the connection path starting with this edge (in opposite orientation), following the horizontal edge labelled $\pm \beta$ in the picture above. This connection path is necessarily a complex $2$-face of the graph, and hence either a triangle, a square or a $6$-gon, with standard labels. 
Note that this complex $2$-face is a signed graph with labels $\lambda-\alpha$ and $\beta$ at $p_1$, with the same complex structures at $p_1$ and $p_2$ as in the picture above.

The next edge in this connection path emerging from $p_2$, say $f$, is hence labelled either $\lambda-\alpha-\beta$, $\lambda-\alpha$ or $\lambda-\alpha+\beta$, corresponding to a complex triangle, square or a $6$-gon.

We claim that all three possibilities contradict the GKM$_3$ condition at $p_2$. Namely, in case the edge labelled $\pm \beta$ is contained in a quaternionic biangle, the quaternionic weight at $p_2$ quaternionically compatible with $-\beta$ is $\lambda-2\beta$ by Figure \ref{fig:quatbiangles}, and thus at $p_2$ there emerge edges with labels $-\beta,\lambda-\beta, \alpha+\beta $ as well as $\lambda-\alpha-3\beta$. This shows that the label of $f$ is a linear combination of $\lambda-\alpha-3\beta$ and $-\beta$ with both coefficients nonzero (it cannot be $-\beta$ as the path cannot return along the same edge at $v_2$).

In case the edge labelled $\pm \beta$ was contained in a quaternionic triangle, the quaternionic weight at $p_2$ quaternionically compatible with $-\beta$ is $\lambda-\beta$ by Figure \ref{fig:actionsonCP2withsign}, hence at $p_2$ there are edges labelled $-\beta, \lambda, \alpha+\beta$ and $\lambda-\alpha-2\beta$. Again, the label of $f$ is a linear combination of $\lambda-\alpha-2\beta$ and $-\beta$ with both coefficients nonzero.
\end{proof}

\section{The Wolf spaces $\HH P^n$ and $\Gr_2(\CC^n)$}\label{sec:thewolfspaces}

Recall that a Wolf space is a quaternion-K\"ahler manifold which at the same time is a Riemannian symmetric space. It was shown in \cite[Section 5]{Wolf} that they are exactly the homogeneous spaces $M=G/H$, where $G$ is a compact simple Lie group, and the isotropy group $H$ is obtained in the following way: after fixing an ordering of the roots of $G$, we have $H=K\cdot {\Sp(1)}$, where $\Sp(1)\cong \SU(2)$ is associated to the highest root, and $K$ the centralizer of $\Sp(1)$ in $G$. The $G$-equivariant subbundle $Q\subset \End(TM)$ is, in the origin $eH$, given by $Q_{eH} = \{\ad_X\mid X\in {\mathfrak{sp}}(1)\}$. As every Wolf space is a homogeneous space of compact Lie groups of equal rank, by \cite{GHZ} the action of a maximal torus $T$ in the isotropy group by left multiplication is of GKM type. Note that the weight of the induced $T$-representation on $Q$ at the origin is, up to sign, given by the highest root.

We will focus on the two Wolf spaces $\HH P^n = \Sp(n+1)/\Sp(1)\times \Sp(n)$ and $\Gr_2(\CC^n)=\Un(n)/\Un(2)\times \Un(n-2)$, as for these spaces the isotropy action is even GKM$_3$.

\subsection{$\HH P^n$}

Consider the Lie group $\Sp(n+1)$ with standard diagonal maximal torus

\[
 T=\left\{\left.\begin{pmatrix}
    e^{i\varphi_0} & & \\
    & \ddots & \\
    & & e^{i\varphi_{n}}
  \end{pmatrix}\right| \varphi_0,\ldots,\varphi_n\in \RR\right\}.
\] Its root system is $\{\pm e_i\pm e_j\mid i,j=0,\ldots,n, i\neq j\}\cup \{\pm 2e_i\mid i=0,\ldots,n\}$. Choosing $2e_0$ as the highest root, we obtain the block-diagonally embedded isotropy $H = \Sp(1)\times \Sp(n)$ of the Wolf space $\HH P^n=\Sp(n+1)/\Sp(1)\times \Sp(n)$.

Observe that the fixed points of the $T$-action on $\HH P^n$ by left multiplication are 
\[
 v_0=[1:0:\ldots:0], v_1=[0:1:0:\ldots:0], \dots v_n=[0:\ldots :0:1];
\]
the point $v_0$ is our chosen origin for the description of ${\mathbb{H}} P^n$ as a homogeneous space. The $1$-skeleton of the action is a union of $\mathbb{H}P^1$'s, given by those points for which at most two of the homogeneous coordinates do not vanish. 

Its GKM graph is a complete graph on the $n+1$ vertices $v_i$, with any two edges connected by a biangle, corresponding to one of the $\mathbb{H}P^1\subset \mathbb{H}P^n$. The weights of the isotropy representation at $v_l$ corresponding to the edges connecting $v_l$ with $v_m$ are given by
\[
\alpha_{lm} =\pm(e_l-e_m), \;\alpha'_{lm}= \pm( e_l+e_m), 
\]
for every $m\in\{0,\dots,n\}, m\neq l$ (see for example \cite[Section 4]{GW}):
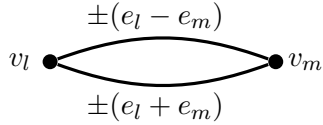
\begin{center}
		\begin{tikzpicture}

			\node (a) at (0,0)[circle,fill,inner sep=2pt] {};
			\node (b) at (3,0)[circle,fill,inner sep=2pt]{};
			\node at (1.4,.6) {$\pm(e_l-e_m)$};
			\node at (1.4,-.6) {$\pm(e_l+e_m)$};
			\node at (-.4,0){$v_l$};
			
			\node at (3.4,0){$v_m$};

			\draw [very thick](a) to[in=160, out=20] (b);
			\draw [very thick](a) to[in=200, out=-20] (b);
		\end{tikzpicture}
	\end{center}
In particular, the action is of type GKM$_3$. The quaternionic weight at $v_0$ is $\pm 2e_0$, equal to the highest root. As the action is GKM$_3$, by Remark \ref{rem:signcompatibilitiesintheunsignedgraph} the unsigned weights determine the quaternionic pairs and quaternionically compatible lifts. Explicitly, the edges forming a biangle are quaternionic pairs; choosing the sign $2e_0$ for the quaternionic weight at $v_0$, the quaternionically compatible lifts are $e_0-e_m$ and $e_0+e_m$, $m=1,\ldots,n$.

By Example \ref{ex:quatbiangle}, see Figure \ref{fig:quatbiangles}, the quaternionic weight at $v_m$ is $\pm 2e_m$; the quaternionic pairs and quaternionically compatible lifts are analogous to those at $v_0$.

In the following proposition, we generalize this example slightly. For the proof of the main result, it will be convenient to rewrite the labels in terms of the weights at the origin. Note that the center ${\mathbb{Z}}_2 = \{\pm 1\}$ of $\Sp(n+1)$ acts trivially on $\HH P^n$.

\begin{proposition}\label{prop:labelsHPn}
Consider the quaternion-K\"ahler manifold $\HH P^n = \Sp(n+1)/\Sp(1)\times \Sp(n)$, with a GKM action of a torus $T$ given by a homomorphism $T\to \Sp(n+1)/\ZZ_2$ into the image of the standard diagonal torus. Let $v_0,\ldots,v_n$ be the $n+1$ fixed points of the action as before, and denote the weight of the induced $T$-representation on $Q_{v_0}$ by $\pm\lambda$. Then we find elements $\alpha_1,\ldots,\alpha_{n}\in \mft^*$ in the integer lattice such that the weights of the edges are given as follows:
\begin{enumerate}
\item For any $k=1,\ldots,n$ there is a biangle between $v_0$ and $v_k$ with labels $\pm \alpha_k$ and $\pm (\lambda-\alpha_k)$.
\item For any $k,l=1,\ldots,n$, $k\neq l$, there is a biangle between $v_k$ and $v_l$ with labels $\pm (\alpha_k-\alpha_l)$ and $\pm (\lambda-\alpha_k-\alpha_l)$.
\end{enumerate}
Moreover, the weights of the induced representations on $Q$ are given as follows:
\begin{enumerate}
\item $\pm \lambda$ at $v_0$
\item $\pm (\lambda-2\alpha_k)$ at $v_k$, for $k=1,\ldots,n$.
\end{enumerate}
Conversely, given characters $\lambda, \alpha_1,\ldots,\alpha_n$ of $T$ (whose derivatives in $\mft^*$ we denote by the same letters) such that the graph $\Gamma$ with edges and labels as above satisfies that for every vertex $v$, the weights of any two edges at $v$ are linearly independent, then there exists a homomorphism of $T$ into the $\ZZ_2$-quotient of the diagonal torus in $\Sp(n+1)$ such that the induced action on $\HH P^n$ is GKM and has GKM graph $\Gamma$.
\end{proposition}

\begin{proof}
In the situation of the example above, if the action is given by a homomorphism into the diagonal torus in $\Sp(n+1)$, all weights are just pullbacks along the homomorphism $\varphi:T\to \Sp(n+1)$. So  we put $\lambda:=\varphi^*(2e_0)$ and $\alpha_k:=\varphi^*(e_0-e_m)$ and rewrite the weights in terms of these linear forms. For the general case, consider a $\ZZ_2$-covering $\tilde{T}$ of $T$ which maps into the diagonal torus of $\Sp(n+1)$. It acts in GKM fashion, with GKM graph as claimed. Then pass to the quotient again.

For the converse, consider the homomorphism of $T$ into the $\ZZ_2$-quotient of the diagonal torus of $\Sp(n+1)$ whose derivative is given by $(\lambda/2,\lambda/2-\alpha_1,\ldots,\lambda/2-\alpha_n)$. Note that the integer lattice of this quotient torus is spanned by $e_0,\ldots,e_{n-1}, \frac{1}{2}(e_0+ \ldots +e_{n})$, so that this homomorphism is well-defined. Its induced action on $\HH P^n$ is GKM, with GKM graph the given one.\footnote{\cite[Proposition 5.13]{GW} states that any GKM$_3$ labeling on the complete graph, with any two edges connected by a biangle, is that of a torus action on $\HH P^n$.  While this statement is true, its proof is slightly incorrect: in the proof it is claimed that all occurring actions are those from \cite[Section 4]{GW}, but as the weights are divided by $2$ as in the argument here, the proof therefore only produces an action of a subtorus of $\Sp(n+1)/\ZZ_2$, not necessarily of $\Sp(n+1)$. \label{footnote} }
\end{proof}
In case $n=1$ we obtain the picture in Figure \ref{fig:quatbiangles}. For $n=2$ we obtain Figure \ref{fig:actionsonHP2}. At every vertex we depicted one possible choice of quaternionically compatible lifts.

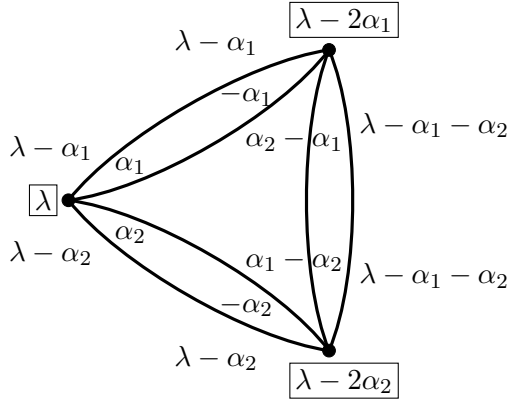
\begin{figure}[h]
\begin{center}
\begin{tikzpicture}[scale=2]
\draw[very thick, rotate around={-60:(2,1)}]
     % (2,1) .. controls (1.93,0.5) and (1.93,-0.5) .. (2,-1)
     (2,1) .. controls (1.8,0.5) and (1.8,-0.5) .. (2,-1)
%        (2,1) .. controls (2.08,0.5) and (2.08,-0.5) .. (2,-1)
       (2,1) .. controls (2.21,0.5) and (2.21,-0.5) .. (2,-1);
 \draw[very thick, rotate around={60:(2,-1)}]
      %(2,1) .. controls (1.93,0.5) and (1.93,-0.5) .. (2,-1)
      (2,1) .. controls (1.8,0.5) and (1.8,-0.5) .. (2,-1)
       %(2,1) .. controls (2.08,0.5) and (2.08,-0.5) .. (2,-1)
       (2,1) .. controls (2.21,0.5) and (2.21,-0.5) .. (2,-1);      
\draw[very thick] %(2,1) .. controls (1.93,0.5) and (1.93,-0.5) .. (2,-1)
      (2,1) .. controls (1.8,0.5) and (1.8,-0.5) .. (2,-1)
       %(2,1) .. controls (2.08,0.5) and (2.08,-0.5) .. (2,-1)
       (2,1) .. controls (2.21,0.5) and (2.21,-0.5) .. (2,-1);
       \node at (0.1,0) [rectangle, draw, inner sep =2pt]{$\lambda$};
        \node at (0.15,0.35) {$\lambda-\alpha_1$};
\filldraw (0.267949192,0) circle (1.2pt)
          (2,1) circle (1.2pt)
          (2,-1) circle (1.2pt);
          
        \node at (.68,.23) {$\alpha_1$};  
        \node at (.68,-.23) {$\alpha_2$};  
        \node at (0.15,-.35) {$\lambda-\alpha_2$};  
       \node at (2.1,1.2) [rectangle, draw, inner sep =2pt]{$\lambda-2\alpha_1$};
       \node at (2.1,-1.2) [rectangle, draw, inner sep =2pt]{$\lambda-2\alpha_2$};
       \node at (1.45,.7) {$-\alpha_1$};
       \node at (1.45,-.7) {$-\alpha_2$};
       \node at (1.25,1.05) {$\lambda-\alpha_1$};
       \node at (1.25,-1.05) {$\lambda-\alpha_2$};
		\node at (1.77,.4) {$\alpha_2-\alpha_1$};
		\node at (1.77,-.4) {$\alpha_1-\alpha_2$};
		\node at (2.7,.5) {$\lambda-\alpha_1-\alpha_2$};
		\node at (2.7,-.5) {$\lambda-\alpha_1-\alpha_2$};

\end{tikzpicture}
\end{center}
\caption{GKM graphs of GKM$_3$ actions on $\HH P^2$}\label{fig:actionsonHP2}
\end{figure}

\subsection{$\Gr_2(\CC^n)$}

Consider the Lie group $\SU(n)$ with standard maximal torus $T$. Its root system is $\{\pm (e_i-e_j)\}$. Choosing $e_1-e_2$ as the highest root, we obtain the isotropy $H = {\mathrm{S}}(\Un(2)\times \Un(n-2))$ of the Wolf space $\Gr_2(\CC^n) = \SU(n)/{\mathrm{S}}(\Un(2)\times \Un(n-2)) = \Un(n)/\Un(2)\times \Un(n-2)$.

Consider the $T$-action on $\Gr_2(\CC^n)$ by left multiplication. The fixed points of this action are the complex $2$-planes $$v_{ij}=v_{ji}=\mathbb{C}_i\oplus \mathbb{C}_j\subset \oplus_{l=1}^n\mathbb{C}_l= \CC^n$$
spanned by the $i$-th and $j$-th standard basis vectors, $i\neq j$. We choose as representatives in $\Gr_2(\mathbb{C}^n)$ the permutation matrices $A_{ij}$ in $\Un(n)$ that switch the first with the $i$-th and the second with the $j$-th standard basis vector.

It is known, see for instance \cite[Proposition 4.1.1]{ChenHe}, that this action is of GKM type, and that the GKM graph is as follows: there are no multiple edges; two vertices $v_{ij}$, $v_{lm}$ are connected by an edge if and only if the sets of indices $\{i,j\}$, $\{l,m\}$ intersect in exactly one element. If $k\in \{1,\dots,n\}, k\neq i,j$, then $v_{ij}$ and $v_{ik}$ are connected by an edge with label $\pm (e_k-e_j)$, and $v_{ij}$ and $v_{jk}$ are connected  with an edge with label $\pm (e_k-e_i)$. Note that in \cite{ChenHe} the GKM graph was determined as a signed graph, with the signs of the labels induced from the natural global complex structure on the Hermitian symmetric space $\Gr_2(\mathbb{C}^n)$. In our situation, however, this complex structure is irrelevant, as this complex structure has nothing to do with the structure of a quaternion-K\"ahler manifold, cf.\ \cite[Proposition 2.6 and Example 1.2]{Pontecorvo}.

The quaternionic weight at the origin $v_{12}$ is the highest root $\pm (e_1-e_2)$. Choosing the lift $e_1-e_2$, the quaternionically compatible lifts of the edges adjacent to $v_{12}$ are $e_1-e_k$ and $e_k-e_2$, $k\geq 3$.

To compute the quaternionic weight at the fixed point $v_{ij}=A_{ij}v_{12}$, we observe that conjugation by $A_{ij}$ of any element in ${\mathfrak{t}}$ permutes the first and second diagonal entries with the $i$-th and $j$-th diagonal entries, respectively. Hence, we obtain that the quaternionic weight at $v_{ij}$ is $\pm (e_i-e_j)$.

As in the case of $\HH P^n$, we formulate a proposition in which we consider a slightly more general action, and rewrite the labels in terms of the weights at the origin.

\begin{proposition}\label{prop:actionsonG2Cn}
Consider the quaternion-K\"ahler manifold $\Gr_2(\CC^n)=\Un(n)/\Un(2)\times \Un(n-2)$, with a GKM action of a torus $T$ given by a homomorphism $T\to \Un(n)$ into the standard diagonal torus of $\Un(n)$. Let $v_{ij}=v_{ji} = \CC_i\oplus\CC_j$ be the $T$-fixed points as before, and denote the weight of the induced $T$-representation on $Q_{v_{12}}$ by $\lambda\in \mft^*$ (with arbitrarily chosen sign). Then we find elements  $\alpha_k\in \mft^*$, $k=3,\ldots,n$, in the integer lattice such that the edges of the GKM graph of the action, together with their labeling, are given as follows:
\begin{enumerate}
\item For any $k=3,\ldots,n$, there is an edge between $v_{12}$ and $v_{1k}$ labeled $\pm \alpha_k$, and an edge between $v_{12}$ and $v_{2k}$ labeled $\pm (\lambda-\alpha_k)$.
\item For any $k=3,\ldots,n$, there is an edge between $v_{1k}$ and $v_{2k}$ labeled $\pm \lambda$.
\item For any $k=1,\ldots,n$ and $l,m=3,\ldots,n$, all three pairwise distinct, there is an edge between $v_{kl}$ and $v_{km}$, labeled $\pm (\alpha_l-\alpha_m)$.
\item For any $k,l=3,\ldots,n$, $k\neq l$, there is an edge between $v_{1k}$ and $v_{kl}$ labeled $\pm (\lambda-\alpha_l)$, and an edge between $v_{2k}$ and $v_{kl}$ labeled $\pm \alpha_l$.
\end{enumerate}
Moreover, the weights of the induced representations on $Q$ are given as follows:
\begin{enumerate}
\item $\pm \lambda$ at $v_{12}$
\item $\pm (\lambda-\alpha_k)$ at $v_{1k}$, $k\geq 3$
\item $\pm \alpha_k$ at $v_{2k}$, $k\geq 3$
\item $\pm (\alpha_k-\alpha_l)$ at $v_{kl}$, $k,l\geq 3$, $k\neq l$.
\end{enumerate}
Conversely, given characters $\lambda,\alpha_3,\ldots,\alpha_n$ of $T$ (whose derivatives in $\mft^*$ we denote by the same letters) such that the graph $\Gamma$ with edges and labels as above satisfies that for every vertex $v$, the weights of any two edges at $v$ are linearly independent, then there exists a homomorphism of $T$ into the diagonal torus of $\Un(n)$ such that the induced action on $\Gr_2(\CC^n)$ is GKM and has GKM graph $\Gamma$.
\end{proposition}
For better illustration, let us draw the  graphs for $n=3$ and $n=4$: for $n=3$, i.e., an action on $\Gr_2(\CC^3) \cong \CC P^2$, we obtain (omitting every $\pm$) Figure \ref{fig:actionsonCP2}.
%\begin{center}
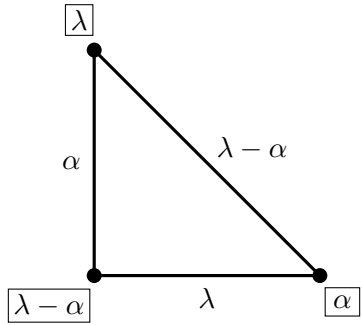
\begin{figure}[h]
		\begin{tikzpicture}

			\draw [very thick] (0,0) -- (3,-3) -- (0,-3) -- (0,0);	
			\node at (0,0)[circle,fill,inner sep=2pt] {};
			\node at (3,-3)[circle,fill,inner sep=2pt]{};
			\node at (0,-3)[circle,fill,inner sep=2pt]{};
			\node at (2.1,-1.3) {$\lambda-\alpha$};
			\node at (-.2,.4)[rectangle, draw, inner sep =2pt]{$\lambda$};
			\node at (-.3,-1.5) {$\alpha$};

%			\node at (-.5,-2.7) {$-\alpha$};
			\node at (-.6,-3.4) [rectangle, draw, inner sep=2pt]{$\lambda-\alpha$};
			\node at (1.5,-3.3) {$\lambda$};
			
%			\node at (2.6,-3.3) {$\lambda$};
			\node at (3.3,-3.35) [rectangle, draw, inner sep = 3pt]{$\alpha$};
%			\node at (3.25,-2.45) {$\alpha-\lambda$};

		\end{tikzpicture}
		\caption{GKM graphs of GKM actions on $\Gr_2(\CC^3)\cong \CC P^2$}\label{fig:actionsonCP2}
		\end{figure}
%	\end{center}
Recall again that because of Remark \ref{rem:signcompatibilitiesintheunsignedgraph}, the information on the quaternionic sign compatibilities is encoded in the data of the unsigned graph and the unsigned weights on $Q$. One possible choice of signs is given as in Figure \ref{fig:actionsonCP2withsign}.
In every vertex, we would be allowed to replace every weight (including that of the representation on $Q$) by its negative. 

For $n=4$ we obtain the unsigned graph in Figure \ref{fig:actionsonG2Cn}. We draw $v_{12}$ on the top, the middle layer has the vertices $v_{13}$ and $v_{14}$ on the left, and $v_{23}$, $v_{24}$ on the right; on the bottom we find $v_{34}$.
\begin{figure}[h]
		\begin{tikzpicture}[scale=1.5]

			\draw [very thick](0,0) -- (3,0)--(2,-1)--(-1,-1)--(0,0);
			\draw [very thick](0,0) -- (1,1.7) -- (3,0);
			\draw [very thick](2,-1) -- (1,1.7)-- (-1,-1);
			\draw [very thick](0,0) -- (1,-2.6) -- (3,0);
			\draw [very thick](2,-1) -- (1,-2.6)-- (-1,-1);
			\node at (0,0)[circle,fill,inner sep=2pt] {};
			\node at (3,0)[circle,fill,inner sep=2pt]{};
			\node at (2,-1)[circle,fill,inner sep=2pt]{};
			\node at (-1,-1)[circle,fill,inner sep=2pt]{};
			\node at (1,-2.6)[circle,fill,inner sep=2pt]{};
			\node at (1,1.7)[circle,fill,inner sep=2pt]{};
			\node at (0.55,-0.2)[rectangle, draw, inner sep =2pt]{$\lambda-\beta$};
			\node at (-0.1,0.5){$\alpha$};
			\node at (-0.3,-2){$\lambda-\beta$};
			\node at (2,-1.5){$\alpha$};
			\node at (1.4,-1.7){$\beta$};
			\node at (0.2,-1.5){$\lambda-\alpha$};
			\node at (-1.5,-1)[rectangle, draw, inner sep =2pt]{$\lambda-\alpha$};
			\node at (1.7,-0.8)[rectangle, draw, inner sep =2pt]{$\alpha$};
			\node at (0.8,-0.8){$\lambda$};
			\node at (1.2,-0.15){$\lambda$};
			\node at (-0.2,-0.6){$\alpha-\beta$};
			\node at (2.3,-0.3){$\alpha-\beta$};
			\node at (1.8,0.5){$\lambda-\alpha$};
			\node at (1,2)[rectangle, draw, inner sep =2pt]{$\lambda$};
			\node at (3.3,0)[rectangle, draw, inner sep =2pt]{$\beta$};
			\node at (1,-2.9)[rectangle, draw, inner sep =2pt]{$\alpha-\beta$};
	\node at (2.3, 1) {$\lambda-\beta$};
	\node at (0.65, 0.8) {$\beta$};
\end{tikzpicture}
\caption{GKM graphs of GKM$_3$ actions on $\Gr_2(\CC^4)$}\label{fig:actionsonG2Cn}
\end{figure}
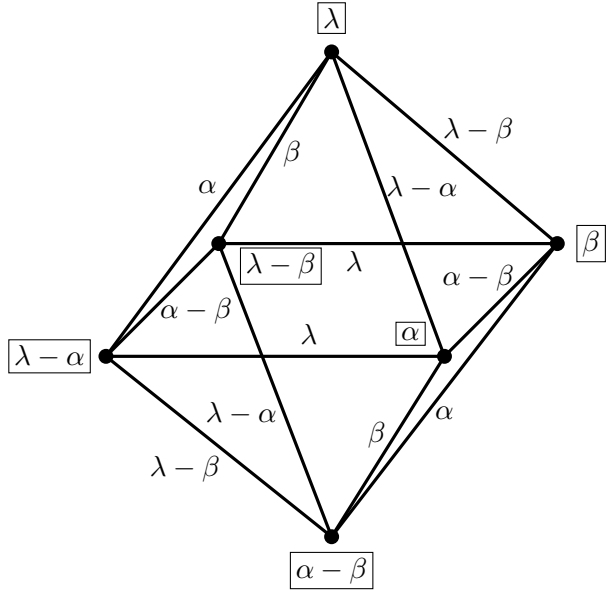
Again, the sign compatibilities are given implicitly by this data.

\begin{proof}
Consider the action given by a homomorphism $T\to \Un(n)$ into the standard diagonal torus. Then the labels are just given by the pullback of the labels of the action of the standard torus along this homomorphism. We now put $\lambda:=\varphi^*(e_1-e_2)\in \mft^*$ and $\alpha_k:=\varphi^*(e_k-e_2)$ for $k=3,\ldots,n$ (the label of the edge $v_{12},v_{1k}$). Writing all the other weights in terms of these linear forms we arrive at the expressions in the statement of the proposition. 

For the converse direction, starting with characters $\lambda$ and $\alpha_3,\ldots,\alpha_n$, we define a homomorphism $T\to \Un(n)$ as $t\mapsto \diag(\lambda(t),1,\alpha_3(t),\ldots,\alpha_n(t))$. If the characters satisfy the given linear independence conditions, the induced action on $\Gr_2(\CC^n)$ is GKM and has the desired GKM graph.
\end{proof}

\section{The main result}

In this section we prove the following purely combinatorial theorem. The assumptions on $\Gamma$ are motivated by the necessary restrictions on the $2$-faces in the quaternion-K\"ahler setting obtained in Proposition \ref{prop:4dimsubmanifolds} and Lemma \ref{lem:no6gon}. Recall the distinction between complex and noncomplex triangles with quaternionic structure from Example \ref{ex:typesoftriangles}.

\begin{theorem}\label{thm:maincombinatorialthm}
Let $\Gamma$ be an abstract GKM$_3$ graph with a quaternionic struture, such that all of its quaternionic $2$-faces are biangles or noncomplex triangles, and that the complex $2$-faces are triangles or quadrangles. Then $\Gamma$ is the GKM graph of a torus action on $\HH P^n$ or $\Gr_2(\CC^n)$ by quaternion-K\"ahler automorphisms.
\end{theorem}

\begin{remark}
Note that due to Proposition \ref{prop:2facescomplex} the complex $2$-faces can a priori never be biangles, as a biangle does not admit any signed structure.
\end{remark}
In the following, we fix an abstract GKM$_3$ graph $\Gamma$ with a quaternionic structure satisfying the assumptions of Theorem \ref{thm:maincombinatorialthm}.

\begin{lemma}\label{lemma:biangles}
If an edge in a complex triangle in $\Gamma$ is contained in a quaternionic biangle, then all three of its edges are contained in a quaternionic biangle.
\end{lemma}
\begin{proof}
Let $p_1,p_2,p_3$ be the three vertices of the complex triangle, where $p_1$ and $p_2$ are connected by a quaternionic biangle. Note that by Proposition \ref{prop:2facescomplex} this triangle admits a signed structure compatible with the quaternionic structure on $\Gamma$. 

By Figure \ref{fig:quatbiangles}, the situation is as follows:

	
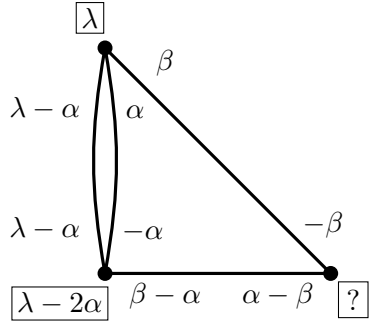
\begin{center}
		\begin{tikzpicture}

			\draw [very thick] (0,0) -- (3,-3) -- (0,-3);
			\draw [very thick] (0,0) to[in=80, out=-80] (0,-3);				\draw [very thick] (0,0) to[in=100, out=-100] (0,-3);
			
			\node at (0,0)[circle,fill,inner sep=2pt] {};
			\node at (3,-3)[circle,fill,inner sep=2pt]{};
			\node at (0,-3)[circle,fill,inner sep=2pt]{};
			\node at (.8,-.2) {$\beta$};
			\node at (-.2,.4)[rectangle, draw, inner sep =2pt]{$\lambda$};
			\node at (.4,-.85) {$\alpha$};
			\node at (-.8,-.8) {$\lambda-\alpha$};

			\node at (.5,-2.45) {$-\alpha$};
			\node at (-.8,-2.4) {$\lambda-\alpha$};
			\node at (-.6,-3.4) [rectangle, draw, inner sep=2pt]{$\lambda-2\alpha$};
			\node at (.8,-3.3) {$\beta-\alpha$};
			
			\node at (2.3,-3.3) {$\alpha-\beta$};
			\node at (3.3,-3.35) [rectangle, draw, inner sep = 3pt]{?};
			\node at (2.9,-2.35) {$-\beta$};

		\end{tikzpicture}
	\end{center}

Assuming that the edge labelled $\pm \beta$ was contained in a quaternionic triangle, then by Figure \ref{fig:actionsonCP2withsign} the weight of $Q_{p_3}$ (indicated with a question mark in the above figure) would be $\lambda-\beta$.  Now consider the remaining edge of the triangle: if it was contained in a quaternionic biangle, then the same weight would be $\lambda-2\alpha - 2(\beta-\alpha) = \lambda-2\beta$ by Figure \ref{fig:quatbiangles}, which is impossible. But if it was contained in a quaternionic triangle, then it would be $\lambda-2\alpha - (\beta-\alpha) = \lambda -\alpha-\beta$ by Figure \ref{fig:actionsonCP2withsign}, which is impossible as well. Hence, the edge labelled $\pm \beta$ cannot be contained in a quaternionic triangle.

Analogously one shows that the edge labelled $\pm (\beta-\alpha)$ cannot be contained in a quaternionic triangle.
\end{proof}

\begin{lemma}\label{lem:kaehlersquares}
If there exists a complex quadrangle $\Gamma'$ in $\Gamma$, then
\begin{enumerate}
\item The labeling of $\Gamma'$ is that of a standard action on $S^2\times S^2$, as in this figure:
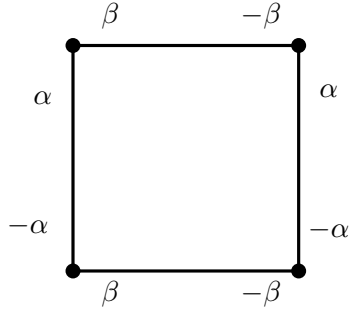
\begin{center}
		\begin{tikzpicture}
			\draw [very thick] (0,0) -- (3,0) -- (3,-3) -- (0,-3)-- (0,0);
			%\draw [very thick] (0,0) to[in=80, out=-80] (0,-3);				\draw [very thick] (0,0) to[in=100, out=-100] (0,-3);
			
			\node at (0,0)[circle,fill,inner sep=2pt] {};
			\node at (3,0)[circle,fill,inner sep=2pt] {};
			\node at (3,-3)[circle,fill,inner sep=2pt]{};
			\node at (0,-3)[circle,fill,inner sep=2pt]{};
			\node at (.5,.4) {$\beta$};
			\node at (2.5,.4) {$-\beta$};
			%\node at (-.2,.4)[rectangle, draw, inner sep =2pt]{$\lambda$};
			\node at (-.4,-.7) {$\alpha$};
			\node at (-.6,-2.4) {$-\alpha$};
			%\node at (-.6,-3.4) [rectangle, draw, inner sep=2pt]{$\lambda-\alpha$};
			\node at (.5,-3.3) {$\beta$};
			
			\node at (2.5,-3.3) {$-\beta$};

			\node at (3.4,-2.45) {$-\alpha$};
			\node at (3.4,-.6) {$\alpha$};
		\end{tikzpicture}
	\end{center}\label{square}
 \item There cannot exist quaternionic biangles in $\Gamma$. \label{biangles}
 \item Any edge of $\Gamma$ is contained in a quaternionic triangle. \label{triangles}
 \item There is no edge in $\Gamma$ between opposite vertices of $\Gamma'$. \label{diag}
 \item Every triangle in $\Gamma$ is a connection triangle.
\end{enumerate}

\end{lemma}
\begin{proof}
By Proposition \ref{prop:2facescomplex} we can fix a signed structure on $\Gamma'$ compatible with the quaternionic structure on $\Gamma$. Enumerate the vertices $p_1,\ldots,p_4$ clockwise, with edges $e_i$ connecting $p_i$ and $p_{i+1}$ (${\mathrm{mod}}\, 4$). Let the signed labels at $p_1$ be $\alpha$ (of $e_4$) and $\beta$ (of $e_1$), and choose the quaternionically compatible lift $\lambda$ of the quaternionic weight.

Then the label of $e_3$ at $p_4$ is of the form $\beta + c\alpha$. The label of $e_2$ is, on the one hand, $\alpha + d\beta$ at $p_2$, and on the other hand $-\alpha + e(\beta + c\alpha)$ at $p_3$. This implies that $ce=0$, and $d=e$. In other words, either $c=0$ or $d=0$, and the other is arbitrary, i.e., the labeling is that of a Hirzebruch surface. Without loss of generality we may  assume that $c=0$, i.e., the graph is as follows:

	
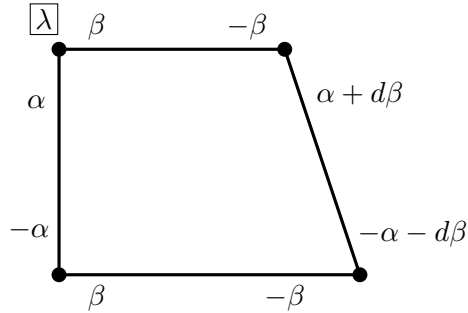
\begin{center}
		\begin{tikzpicture}

			\draw [very thick] (0,0) -- (3,0) -- (4,-3) -- (0,-3);
			\draw [very thick] (0,0) to(0,-3);							
			\node at (0,0)[circle,fill,inner sep=2pt] {};
			\node at (3,0)[circle,fill,inner sep=2pt] {};
			\node at (4,-3)[circle,fill,inner sep=2pt]{};
			\node at (0,-3)[circle,fill,inner sep=2pt]{};
			\node at (.5,.3) {$\beta$};
			\node at (2.5,.3) {$-\beta$};
			\node at (-.2,.4)[rectangle, draw, inner sep =2pt]{$\lambda$};
			\node at (-.3,-.7) {$\alpha$};

			\node at (-.4,-2.4) {$-\alpha$};
			\node at (.5,-3.3) {$\beta$};
			
			\node at (3,-3.3) {$-\beta$};

			\node at (4.7,-2.45) {$-\alpha-d\beta$};
			\node at (4,-.6) {$\alpha+d\beta$};

		\end{tikzpicture}
	\end{center}
In order to show \eqref{square}, we need to have $d=0$.

Assume that the left vertical edge of $\Gamma'$ is contained in a quaternionic biangle (with quaternionically compatible labels $\lambda-\alpha$ both at $p_1$ and $p_2$):
	
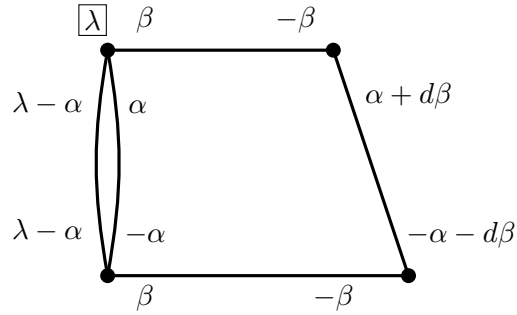
\begin{center}
		\begin{tikzpicture}

			\draw [very thick] (0,0) -- (3,0) -- (4,-3) -- (0,-3);
			\draw [very thick] (0,0) to[in=80, out=-80] (0,-3);				\draw [very thick] (0,0) to[in=100, out=-100] (0,-3);
	
			\node at (0,0)[circle,fill,inner sep=2pt] {};
			\node at (3,0)[circle,fill,inner sep=2pt] {};
			\node at (4,-3)[circle,fill,inner sep=2pt]{};
			\node at (0,-3)[circle,fill,inner sep=2pt]{};
			\node at (.5,.4) {$\beta$};
			\node at (2.5,.4) {$-\beta$};
			\node at (-.2,.4)[rectangle, draw, inner sep =2pt]{$\lambda$};
			\node at (.4,-.75) {$\alpha$};
			\node at (-.8,-.7) {$\lambda-\alpha$};

			\node at (.5,-2.45) {$-\alpha$};
			\node at (-.8,-2.4) {$\lambda-\alpha$};
			\node at (.5,-3.3) {$\beta$};
			
			\node at (3,-3.3) {$-\beta$};

			\node at (4.7,-2.45) {$-\alpha-d\beta$};
			\node at (4,-.6) {$\alpha+d\beta$};

		\end{tikzpicture}
	\end{center}
 By the GKM$_3$ condition, the connection path starting with the lower horizontal edge, followed by the vertical edge labelled $\lambda-\alpha$, has to continue with the upper horizontal edge. But then, by definition of a quaternionic structure, $\beta$ is quaternionically compatible with $-(\lambda-\alpha)$, which contradicts the fact that it was already quaternionically compatible with $\lambda-\alpha$. Hence this biangle cannot exist. Analogously, the right vertical edge cannot be contained in a quaternionic biangle.

Assume now the upper horizontal edge to be contained in a quaternionic biangle. In case $d=0$ the argument from before applies to obtain a contradiction, so consider the case $d\neq 0$ for the moment:
	
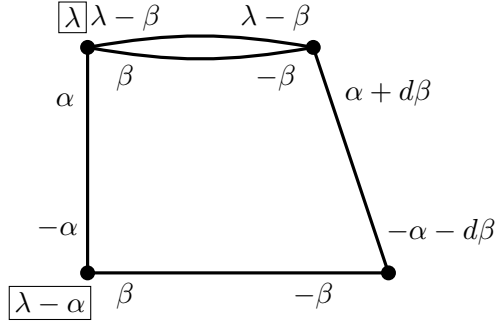
\begin{center}
		\begin{tikzpicture}

			\draw [very thick] (0,0) to[in=170, out=10] (3,0);
			\draw [very thick] (0,0) to[in=190, out=-10] (3,0);
			\draw[very thick] (3,0) -- (4,-3) -- (0,-3);
			\draw [very thick] (0,0) to(0,-3);							
			\node at (0,0)[circle,fill,inner sep=2pt] {};
			\node at (3,0)[circle,fill,inner sep=2pt] {};
			\node at (4,-3)[circle,fill,inner sep=2pt]{};
			\node at (0,-3)[circle,fill,inner sep=2pt]{};
			\node at (.5,.4) {$\lambda-\beta$};
			\node at (2.5,.4) {$\lambda-\beta$};
			\node at (.5,-.4) {$\beta$};
			\node at (2.5,-.4) {$-\beta$};
			\node at (-.2,.4)[rectangle, draw, inner sep =2pt]{$\lambda$};
			\node at (-.3,-.7) {$\alpha$};

			\node at (-.4,-2.4) {$-\alpha$};
\node at (-.5,-3.4)[rectangle, draw, inner sep =2pt]{$\lambda-\alpha$};		
			\node at (.5,-3.3) {$\beta$};
			
			\node at (3,-3.3) {$-\beta$};

			\node at (4.7,-2.45) {$-\alpha-d\beta$};
			\node at (4,-.6) {$\alpha+d\beta$};

		\end{tikzpicture}
	\end{center}
Here, the quaternionic weight at $p_4$ is $\lambda-\alpha$, as we already saw that $e_4$ is necessarily contained in a quaternionic triangle.	
	
Consider now the $2$-face containing the left vertical edge and the horizontal edge labelled $\lambda-\beta$. If it is a triangle, then its third edge is labelled $\lambda-\beta-\alpha$ at $p_4$. But this is exactly the label of the edge that forms a quaternionic pair with $e_3$. Hence, the lower horizontal edge is contained in a quaternionic triangle, whose third edge forms a biangle with the edge labelled $\alpha + d\beta$. This contradicts the fact that there are no complex biangles in $\Gamma$.

If this $2$-face is a quadrangle, then we know from before that one of its pairs of opposite edges has to be labelled identically. By the GKM$_3$ condition it cannot be the pair of vertical edges, as at $p_2$ there cannot exist edges labelled $-\beta$, $\alpha+d\beta$ and $-\alpha$ (we assumed $d\neq 0$). So at $p_4$ there has to be another edge labelled $\lambda-\beta$. But the quaternionic weight at $p_4$ is $\lambda-\alpha$, hence the edge that forms a quaternionic pair with $e_3$ is labelled $\lambda-\alpha-\beta$. This contradicts the GKM$_3$ condition at $p_2$.

The same argument applies to the lower horizontal edge, so that none of the edges of $\Gamma'$ is contained in a quaternionic biangle. Now, $d=0$ follows by transporting the quaternionic weight around $\Gamma'$: at $p_3$ the quaternionic weight is $\lambda-\alpha-\beta$, following the left and lower edge, but at the same time $\lambda - \alpha - (d+1)\beta$, following the upper and right edge. This shows \eqref{square}. Note that as all edges of $\Gamma'$ are contained in quaterionic triangles,  we obtain, using Figure \ref{fig:actionsonCP2withsign}, the following labeling on $\Gamma'$:
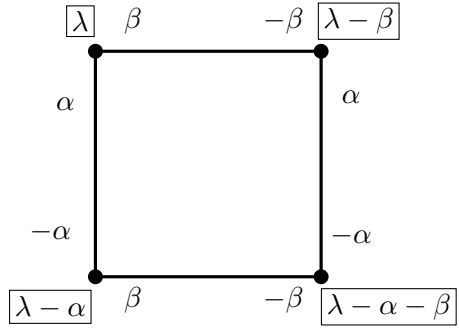
	\begin{center}
		\begin{tikzpicture}
			\draw [very thick] (0,0) -- (3,0) -- (3,-3) -- (0,-3)-- (0,0);
			%\draw [very thick] (0,0) to[in=80, out=-80] (0,-3);				\draw [very thick] (0,0) to[in=100, out=-100] (0,-3);
			
			\node at (0,0)[circle,fill,inner sep=2pt] {};
			\node at (3,0)[circle,fill,inner sep=2pt] {};
			\node at (3,-3)[circle,fill,inner sep=2pt]{};
			\node at (0,-3)[circle,fill,inner sep=2pt]{};
			\node at (.5,.4) {$\beta$};
			\node at (2.5,.4) {$-\beta$};
			\node at (-.2,.4)[rectangle, draw, inner sep =2pt]{$\lambda$};
			\node at (3.5,.4)[rectangle, draw, inner sep =2pt]{$\lambda-\beta$};
			\node at (-.4,-.7) {$\alpha$};
			\node at (-.6,-2.4) {$-\alpha$};
			\node at (-.6,-3.4) [rectangle, draw, inner sep=2pt]{$\lambda-\alpha$};
			\node at (3.9,-3.4) [rectangle, draw, inner sep=2pt]{$\lambda-\alpha-\beta$};
			\node at (.5,-3.3) {$\beta$};
			
			\node at (2.5,-3.3) {$-\beta$};

			\node at (3.4,-2.45) {$-\alpha$};
			\node at (3.4,-.6) {$\alpha$};
		\end{tikzpicture}
	\end{center}
	
	\vspace{0.1 cm}
	Now we assume there exists a quaternionic biangle anywhere in $\Gamma$. Consider a simple path (namely a path without repeating vertices) with edges $[e_1, \dots, e_m]$ connecting a vertex of the given biangle and a vertex of $\Gamma'$, none of them being an edge of $\Gamma'$ as well as of the considered biangle. As showed above, the connection path starting with an edge of the biangle and followed by $e_1$ is --by the GKM$_3$ condition-- a complex triangle whose edges are all contained in biangles by Lemma \ref{lemma:biangles}. Iterating this argument we have that every edge $e_i$ is contained in a quaternionic biangle. In particular the connection path through $e_m$ and an adjacent edge of $\Gamma'$ is a complex triangle with all edges contained in quaternionic biangles. This contradicts the first part of this proof, and hence parts  \eqref{biangles} and \eqref{triangles} are proven.

To show part \eqref{diag} we assume there is an edge $v$ in $\Gamma$ between two opposite vertices of $\Gamma'$. By the GKM$_3$ condition and the assumptions, any connection path through an edge of $\Gamma'$ and the diagonal edge $v$ is a (twisted) complex square: it cannot be a triangle, as this would result in a biangle in the graph, which is ruled out by part \eqref{biangles}. Thus we obtain the following signed labelled graph:
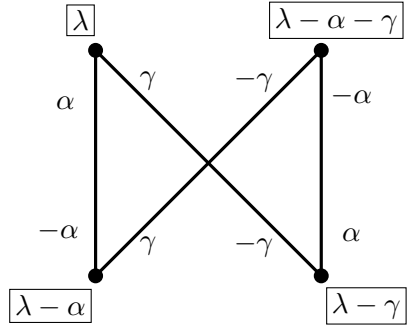
\begin{center}
		\begin{tikzpicture}

			\draw [very thick](0,0) -- (0,-3);
			\draw [very thick](3,0) -- (3,-3);
			%\draw  [ultra thin](0,0) -- (3,0) -- (3,-3) -- (0,-3) -- (0,0);
			\draw [very thick] (0,-3) -- (3,0);
			\draw [very thick] (0,0) -- (3,-3);
			\node at (0,0)[circle,fill,inner sep=2pt] {};
			\node at (3,0)[circle,fill,inner sep=2pt] {};
			\node at (3,-3)[circle,fill,inner sep=2pt]{};
			\node at (0,-3)[circle,fill,inner sep=2pt]{};
% 			\node at (.5,.4) {$\beta$};
% 			\node at (2.5,.4) {$-\beta$};
			\node at (-.2,.4)[rectangle, draw, inner sep =2pt]{$\lambda$};
			\node at (3.2,.4)[rectangle, draw, inner sep =2pt]{$\lambda-\alpha-\gamma$};
			\node at (-.4,-.7) {$\alpha$};
            \node at (.7,-.4) {$\gamma$};
            \node at (.7,-2.6) {$\gamma$};
            \node at (2.1,-.4) {$-\gamma$};
            \node at (2.1,-2.6) {$-\gamma$};
			\node at (-.5,-2.4) {$-\alpha$};
			\node at (-.6,-3.4) [rectangle, draw, inner sep=2pt]{$\lambda-\alpha$};
			\node at (3.6,-3.4) [rectangle, draw, inner sep=2pt]{$\lambda-\gamma$};
% 			\node at (.5,-3.3) {$\beta$};
% 			
% 			\node at (2.5,-3.3) {$-\beta$};

			\node at (3.4,-2.45) {$\alpha$};
			\node at (3.4,-.6) {$-\alpha$};
\end{tikzpicture}
	\end{center}

 Comparing the labels of the last two pictures at the right down corner, we have that
\begin{equation}\label{eq:twisted square}
  \lambda - \gamma= -\lambda+\alpha+\beta. 
\end{equation}
We consider the following three edges emerging from the lower left vertex $p_4$ of $\Gamma'$:
\begin{itemize}
\item the second edge of the quaternionic triangle containing the diagonal $[p_4,p_2]$,
\item the second edge of the quaternionic triangle containing the lower horizontal edge of $\Gamma'$
 \item the left vertical edge of $\Gamma'$
\end{itemize}
which have respectively the following labeling
\[
 \lambda-\alpha-\gamma, \; \lambda-\alpha-\beta, \;-\alpha.
\]
By \eqref{eq:twisted square}, the first of these weights is $-\lambda+\beta$, hence we have obtained a contradiction to the GKM$_3$ condition.
	
	Finally, we consider any triangle in $\Gamma$ and the connection path through any two edges of the triangle. By part \eqref{diag} of this lemma we know that it can not be a quadrangle, and hence it has to be a connection triangle. By part \eqref{biangles} of this lemma there are no biangles, so this connection triangle coincides with the considered triangle.

\end{proof}

\begin{proposition}\label{proposition:HPn}
 If $\Gamma$ contains a quaternionic biangle, then it is the graph of a GKM$_3$-action on $\HH P^n$ by quaternion K\"ahler automorphisms.
\end{proposition}

\begin{proof}
 Consider any quaternionic biangle in $\Gamma$. If $\Gamma$ is at least $3$-valent, we have that the connection path through any of its two edges and any other edge emerging from any of the two vertices of this biangle is a complex triangle, as by Lemma \ref{lem:kaehlersquares} above it cannot be a square. By Lemma~\ref{lemma:biangles} every edge of this triangle is contained in a quaternionic biangle. Iterating this argument we have that $\Gamma$ is a complete graph, with any two vertices connected by a quaternionic biangle. 
 
Consider any vertex $v_0$; then, the labels of the biangles at $v_0$ are of the form $\pm \alpha_k$ and $\pm (\lambda-\alpha_k)$, $k=1,\ldots,n$, where $\pm \lambda$ is the quaternionic weight at $v_0$. Denote the second vertices of these biangles by $v_k$. The quaternionic weight at $v_k$ is necessarily $\pm (\lambda-2\alpha_k)$. Consider the complex $2$-face determined by the edges labelled $\pm \alpha_k$ and $\pm \alpha_l$; we choose $\alpha_k$ and $\alpha_l$ as quaternionically compatible lifts at $v_0$, which are then automatically part of a compatible signed labeling of this $2$-face. Then one of the edges at $v_k$ in direction $v_l$ is labelled $\alpha_l-\alpha_k$. This label is quaternionically compatible with the labels $-\alpha_k$ and the lift $\lambda-2\alpha_k$ of the quaternionic weight, hence the other edge from $v_k$ to $v_l$ is labelled $\pm (\lambda-\alpha_l-\alpha_k)$. We determined the whole graph, so we can apply Proposition \ref{prop:labelsHPn} to deduce that $\Gamma$ is the graph of an action on $\HH P^n$ by quaternion-K\"ahler automorphisms.
 \end{proof}
\begin{remark}
By \cite[Lemma 5.13]{GW} it even holds true that any GKM$_3$ labeling on the complete graph with edges doubled is of this type. (But see Footnote \ref{footnote}.)
\end{remark}

\begin{lemma}\label{lemma:4-faces}
Assume that $\Gamma$ does not contain a quaternionic biangle (i.e., all its quaternionic $2$-faces are quaternionic triangles). Let $v_0$ be a vertex of $\Gamma$, and let $\lambda$ be any lift of the quaternionic weight at $v_0$. Consider two quaternionic triangles meeting in $v_0$. Then we find linear forms $\alpha,\beta\in \mft^*$ such that the lifts of the labels of the edges of the two quaternionic triangles which are quaternionically compatible with $\lambda$ are $\alpha,\lambda-\alpha$ and $\beta,\lambda-\beta$, respectively, and such that these triangles are contained in a $4$-valent subgraph as in Figure \ref{fig:actionsonG2Cn} (with $v_0$ as the top vertex). This subgraph is invariant under the connection (i.e., it is a $4$-face), and the quaternionic structure of $\Gamma$ restricts to it (i.e., it is quaternionic).
\end{lemma}
\begin{proof} 
Consider any vertex $v_0$ of $\Gamma$ and two quaternionic pairs of edges $(e_1,e_2), (e_3,e_4)$ emerging from this vertex. These edges have labeling 
 \[
  \alpha, \lambda-\alpha, \beta,\lambda-\beta,
 \]
 respectively, for some $\alpha,\beta\in \mft^*$.  We know that the connection paths through $e_1, e_2$ and through $e_3,e_4$ are quaternionic triangles, which we denote by $[e_1,e_2,e_5]$ and $[e_3,e_4,e_6]$. As the graph does not contain biangles, the two triangles meet only in $v_0$. The connection paths through any other pair of edges in $\{e_1,\dots,e_4\}$ give a complex triangle or a complex square. 
Let us consider first the case that the connection path through $e_1$ and $e_3$ is a complex triangle $[e_1,e_3,e_7]$. As the labeling of both $e_5$ and $e_6$ is $\lambda$, the GKM$_3$ condition implies that the connection path trough $e_5,e_7$ is a complex square continuing with $e_6$; we call its remaining edge $e_8$.

Observe that at each of the end points $v_i$ of $e_i$, $i=1,\dots,4$, we have a pair of quaternionic edges (namely the two edges of one of the two quaternionic triangles considered above), and at each of these vertices $v_i$, there exists another edge $e'_i$ that together with $e_7$ or $e_8$ forms another quaternionic triangle. We are thus in the following situation:

 
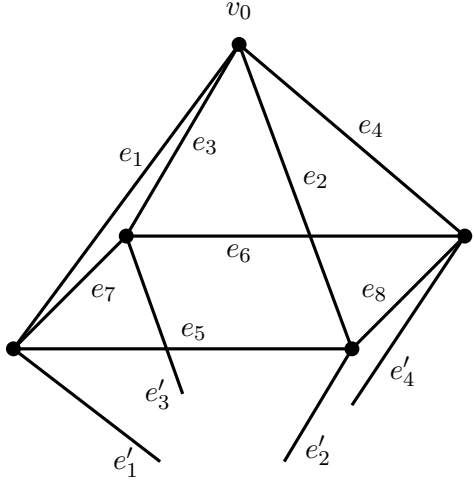
\begin{center}
		\begin{tikzpicture}[scale=1.5]

			\draw [very thick](2,-1)--(-1,-1)--(0,0)--(3,0);
			\draw [very thick](0,0) -- (1,1.7) -- (3,0)--(2,-1);
			\draw [very thick](2,-1) -- (1,1.7)-- (-1,-1);

			\node at (0,0)[circle,fill,inner sep=2pt] {};
			\node at (3,0)[circle,fill,inner sep=2pt]{};
			\node at (2,-1)[circle,fill,inner sep=2pt]{};
			\node at (-1,-1)[circle,fill,inner sep=2pt]{};
			\node at (1,1.7)[circle,fill,inner sep=2pt]{};
		\node at (1,2){$v_0$};
			\node at (1.68,0.5) {$e_2$};
			\node at (0.05,0.7) {$e_1$};
			\node at (0.6,-0.85) {$e_5$};
			\node at (2.17,0.95) {$e_4$};
			\node at (0.7,0.8) {$e_3$};
			\node at (1,-0.15) {$e_6$};
			\node at (-0.2,-0.5) {$e_7$};

			\draw [very thick](0,0) -- (0.5,-1.4);
			\node at (0,-2) {$e'_1$};
			\node at (0.28,-1.4) {$e'_3$};
			\node at (1.7,-1.9) {$e'_2$};
			\node at (2.45,-1.2) {$e'_4$};

			\draw [very thick](2,-1.5) -- (3,0);
			\draw [very thick](2,-1) -- (1.4,-2);
			\draw [very thick](0.3,-2)-- (-1,-1);
			\node at (2.2,-0.5) {$e_8$};

\end{tikzpicture}
	\end{center}
	Using Figure \ref{fig:actionsonCP2withsign}, the labeling of all edges is already as claimed, i.e., coincides with Figure \ref{fig:actionsonG2Cn}. Inspecting the labels and using the GKM$_3$ condition, we obtain that the paths $e_1',e_1,e_4,e_4'$ and $e_2',e_2,e_3,e_3'$ form complex squares, and that the paths $e_1',e_7,e_3'$ and $e_4',e_8,e_2'$ form complex triangles. Hence, the four edges $e_i'$ all meet in the same point, and we have obtained the desired configuration.
	
Finally, if instead we assume that the connection path through $e_1$ and $e_3$ is a square, then the connection path through $e_1,e_4$ is a complex triangle (otherwise at the vertex $v_1$ we would have three edges labeled $\lambda$, $\beta$, and $\lambda-\beta$). Thus we are in the previous situation, with the role of $e_3$ and $e_4$ interchanged. 	
\end{proof}

The following proposition concludes the proof of Theorem \ref{thm:maincombinatorialthm}.

\begin{proposition}\label{prop:quatgraphsnobiangles}
Let $\Gamma$ be a $(2n-4)$-valent abstract GKM$_3$ graph with quaternionic structure such that all quaternionic $2$-faces are  noncomplex triangles, and all complex $2$-faces are triangles or squares. Then $\Gamma$ is the GKM graph of some GKM$_3$-action on $\Gr_2(\CC^{n})$ by quaternion-K\"ahler automorphisms.
\end{proposition}

\begin{proof}
For $n=3$ the statement is obvious, as the whole graph $\Gamma$ consists only of one quaternionic triangle, which is necessarily of the type in Figure \ref{fig:actionsonCP2}. From now on, consider $n\geq 4$.

Consider any vertex of $\Gamma$, and call it $v_{12}$. Let $\lambda\in \mft^*$ be the quaternionic weight at $v_{12}$, lifted arbitrarily. We know that there are $n-2$ quaternionic triangles meeting in the vertex $v_{12}$. By Lemma \ref{lemma:4-faces}, any two of these are contained in a quaternionic $4$-face of the type in Figure \ref{fig:actionsonG2Cn}. Considering the $4$-face containing the first two, we may denote the other vertices of this $4$-face $v_{13}, v_{14}, v_{23}, v_{24}$ and $v_{34}$, and find linear forms $\alpha_3,\alpha_4\in \mft^*$ such that the conditions of Proposition \ref{prop:actionsonG2Cn} that involve only vertices $v_{ij}$ with $1\leq i,j\leq 4$ are satisfied.

Let us consider now the $4$-face containing the second and third quaternionic triangle. We may denote the endpoints of the edges at $v_{12}$ in the third quaternionic triangle by $v_{15}$ and $v_{25}$ in such a way that $v_{14}$ and $v_{15}$, as well as $v_{24}$ and $v_{25}$ are connected by an edge. We denote the remaining vertex of the quaternionic $4$-face by $v_{45}$. Moreover, we find a linear form $\alpha_5\in \mft^*$ such that the conditions of Proposition \ref{prop:actionsonG2Cn} that involve only vertices $v_{ij}$ with $i,j=1,2,4,5$ are satisfied. Note that all the vertices we defined so far are necessarily distinct, as our graph does not contain any biangles.

Now consider the $4$-face spanned by the first and third quaternionic triangle. It contains the vertices $v_{12}, v_{13}, v_{23}, v_{15}, v_{25}$ as well as another vertex (again disjoint from all previously defined ones) that we denote $v_{35}$. We claim that $v_{13}$ and $v_{15}$, as well as $v_{23}$ and $v_{25}$ are connected by an edge. To see this, we show that the complex $2$-face defined by the connection path $v_{13},v_{14},v_{15}$ is a triangle. The label of the edge $v_{13},v_{14}$ is $\pm (\alpha_3-\alpha_4)$, that of the edge $v_{14},v_{15}$ is $\pm (\alpha_4 - \alpha_5)$, and the one of the edge starting at $v_{13}$ in direction either $v_{15}$ or $v_{25}$ is $\pm (\alpha_3-\alpha_5)$. By Proposition \ref{lem:kaehlersquares}, part \eqref{square}, a square is impossible as $\pm (\alpha_3-\alpha_5)\neq \pm (\alpha_4 - \alpha_5)$. All in all, we found the following configuration in our graph, with all labels compatible with Proposition \ref{prop:actionsonG2Cn}.
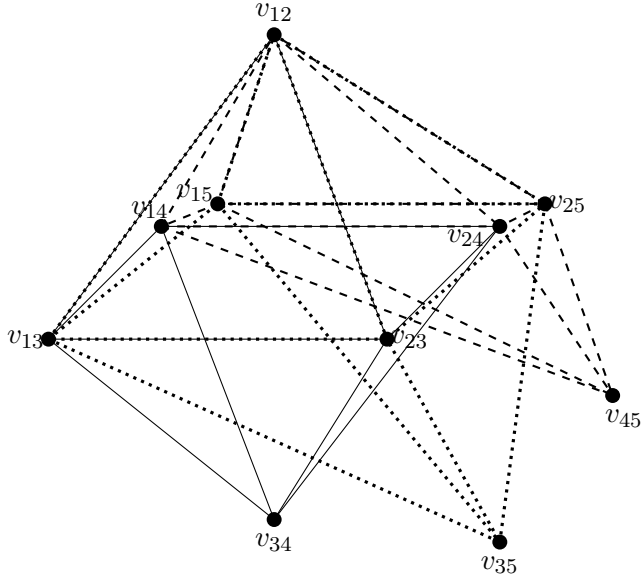
 \begin{center}
		\begin{tikzpicture}[scale=1.5]

			\draw (0,0) -- (3,0)--(2,-1)--(-1,-1)--(0,0);
			\draw[thick, dashed] (0,0) -- (1,1.7) -- (3,0);
			\draw (2,-1) -- (1,1.7)-- (-1,-1);
			\draw (0,0) -- (1,-2.6) -- (3,0);
			\draw (2,-1) -- (1,-2.6)-- (-1,-1);
			\node at (0,0)[circle,fill,inner sep=2pt] {};
			\node at (3,0)[circle,fill,inner sep=2pt]{};
			\node at (2,-1)[circle,fill,inner sep=2pt]{};
			\node at (-1,-1)[circle,fill,inner sep=2pt]{};
			\node at (1,-2.6)[circle,fill,inner sep=2pt]{};
			\node at (1,1.7)[circle,fill,inner sep=2pt]{};
			
			\node at (-0.1,0.15) {$v_{14}$};
			\node at (2.7,-0.1){$v_{24}$};
			\node at (2.2,-1){$v_{23}$};
			\node at (-1.2,-1){$v_{13}$};
			\node at (1,-2.8){$v_{34}$};
			\node at (1,1.9){$v_{12}$};
\draw[thick, dashed] (0,0) -- (3,0)--(3.4,0.2)--(0.5,0.2)--(0,0);
\draw[thick, dashed] (3.4,0.2) -- (1,1.7) -- (0.5,0.2);

\draw[thick, dashed] (3.4,0.2) -- (4,-1.5) -- (0.5,0.2);
\draw[thick, dashed] (0,0) -- (4,-1.5) -- (3,0);
			\node at (3.4,0.2)[circle,fill,inner sep=2pt]{};
			\node at (4,-1.5)[circle,fill,inner sep=2pt]{};
			\node at (0.5,0.2)[circle,fill,inner sep=2pt]{};
			
			\node at (3.6,0.2){$v_{25}$};
			\node at (4.1,-1.7){$v_{45}$};
			\node at (0.3,0.3){$v_{15}$};

\draw[very thick, dotted] (-1,-1)--(1,1.7)--(2,-1);
\draw[very thick, dotted] (0.5,0.2)--(1,1.7)--(3.4,0.2);
\draw[very thick, dotted] (-1,-1)--(0.5,0.2)--(3.4,0.2)--(2,-1)--(-1,-1);
\draw[very thick, dotted] (-1,-1)--(3,-2.8)--(2,-1);
\draw[very thick, dotted] (0.5,0.2)--(3,-2.8)--(3.4,0.2);
\node at (3,-2.8)[circle,fill,inner sep=2pt]{};
\node at (3,-3){$v_{35}$};
\end{tikzpicture}
\end{center}

We continue with this process of adding quaternionic triangles to the picture. In the end we have found almost all the edges of the graph, with linear forms $\alpha_3,\ldots,\alpha_n\in \mft^*$ such that almost all conditions of Proposition \ref{prop:actionsonG2Cn} are satisfies. All that remains to be shown is that for any $k,l,m=3,\ldots,n$, all three pairwise distinct, there is an edge between $v_{kl}$ and $v_{km}$ labeled $\pm (\alpha_l-\alpha_m)$, i.e., a part of Condition (3) in Proposition \ref{prop:actionsonG2Cn}. In the picture above, this is the missing triangle $v_{34}, v_{35},v_{45}$. To find these edges, consider the connection path through $v_{kl},v_{1k},v_{1m}$. The label of the edge $v_{kl},v_{1k}$ is $\pm (\lambda-\alpha_l)$, and that of the edge $v_{1k},v_{1m}$ is $\pm (\alpha_l-\alpha_m)$. It follows from the GKM$_3$ condition that the edge $v_{1m},v_{lm}$ labeled $\pm (\lambda-\alpha_l)$ is the next one in the connection path, which forces this complex $2$-face to be a square. Hence, there exists an edge $v_{kl},v_{lm}$ labeled $\pm (\alpha_l-\alpha_m)$.\end{proof}

As a corollary, we obtain the following statement, which is however already covered by \cite[Theorem 5.1]{Occhetta}.

\begin{corollary}\label{cor:main}
Let $M$ be a positive quaternion-K\"ahler manifold that admits a torus action of type GKM$_3$. Then $M$ has the rational cohomology ring of either $\HH P^n$ or the complex Grassmannian of $2$-planes $\Gr_2(\CC^{n})$.
\end{corollary}

\begin{proof}
Consider the GKM graph $\Gamma$ of a GKM$_3$ action of a torus $T$ on a positive quaternion-K\"ahler manifold $M$. By Proposition \ref{prop:gammahasquat}, $\Gamma$ has a quaternionic structure. By Proposition \ref{prop:4dimsubmanifolds}, its quaternionic $2$-faces are biangles or quaternionic triangles, and by the same proposition and Lemma \ref{lem:no6gon} the other $2$-faces are triangles or squares. Hence, by Theorem \ref{thm:maincombinatorialthm}, $\Gamma$  is necessarily the graph of a GKM action on $\HH P^n$ or $\Gr_2(\CC^{n+2})$ by quaternion-K\"ahler automorphisms. This implies, by Proposition \ref{prop:graphdeteqcohom}, that the rational equivariant cohomology algebra of the $T$-action on $M$ is that of an action on these model spaces. As for equivariantly formal actions, equivariant cohomology determines the ordinary cohomology ring by Proposition \ref{prop:eqformalcomputable}, the claim follows.
\end{proof}

\end{document}